\documentclass[11pt]{amsart}
\usepackage{geometry,verbatim,cite}                % See geometry.pdf to learn the layout options. There are lots.
\usepackage{graphicx}
\usepackage{amssymb}
\usepackage{epstopdf}
\usepackage{graphicx}
\usepackage{amsmath,amscd}
\usepackage{mathrsfs}
\usepackage{caption}
\usepackage{subcaption}
\usepackage{color}
\usepackage{enumerate}
\usepackage{subcaption}

\usepackage[normalem]{ulem}

\title[Momentum Space with Relaxation]{Electronic Observables for Relaxed Bilayer 2D Heterostructures in Momentum Space }

\author[D. Massatt]{Daniel Massatt}
\address{D. Massatt \\ Department of Mathematics \\ Louisiana State University \\ Baton Rouge, Louisiana, 70803 \\ USA.}
\email{dmassatt@lsu.edu}

\author[S. Carr]{Stephen Carr}
\address{S. Carr \\ Department of Physics \\ Brown University \\ Providence, Rhode Island 02912 \\ USA}
\email{stcarr.nj@gmail.com}
\author[M. Luskin]{Mitchell Luskin}
\address{M. Luskin \\ School of Mathematics \\ University of Minnesota \\ Minneapolis, Minnesota, 55455 \\ USA}
\email{luskin@umn.edu}
\thanks{ML's research was supported in part by NSF Award DMS-1906129.}
\date{\today}                                           % Activate to display a given date or no date

\numberwithin{equation}{section}

\begin{document}
\begin{abstract}
Momentum space transformations for incommensurate 2D electronic structure calculations are fundamental for reducing computational cost and for representing the data in a more physically motivating format, as exemplified in the Bistritzer-MacDonald model~\cite{bistritzer2011}. However, these transformations can be difficult to implement in more complex systems such as when mechanical relaxation patterns are present. In this work, we aim for two objectives. Firstly, we strive to simplify the understanding and implementation of this transformation by rigorously writing the transformations between the four relevant spaces, which we denote real space, configuration space, momentum space, and reciprocal space. This provides a straight-forward algorithm  for writing the complex momentum space model from the original real space model. Secondly, we implement this for twisted bilayer graphene with mechanical relaxation affects included. We also analyze the convergence rates of the approximations, and show the tight-binding coupling range increases for smaller relative twists between layers, demonstrating that the 3-nearest neighbor coupling of the Bistritzer-MacDonald model is insufficient when mechanical relaxation is included for very small angles. We quantify this and verify with numerical simulation.
\end{abstract}

\maketitle
%\section{}
%\subsection{}

%\newcommand{\ml}[1]{{\color{red} #1}}
%\newcommand{\cml}[1]{{\small \it \color{red} [ML: #1]}}
%\newcommand{\ml}[1]{{ #1}}
%\newcommand{\cml}[1]{{\small \it \color{red} [ML: #1]}}
%\newcommand{\co}[1]{{\color{blue} #1}}
%\newcommand{\cco}[1]{{\small \it \color{blue} [CO: #1]}}
%\newcommand{\dm}[1]{{ #1}} % normally it is magenta
%\newcommand{\stc}[1]{{ #1}}
\newcommand{\dm}[1]{{ #1}} % normally it is magenta
\newcommand{\stc}[1]{{ #1}}
\newcommand{\refchange}[1]{{\color{black} #1}}

% mean integral
\def\Xint#1{\mathchoice
{\XXint\displaystyle\textstyle{#1}}%
{\XXint\textstyle\scriptstyle{#1}}%
{\XXint\scriptstyle\scriptscriptstyle{#1}}%
{\XXint\scriptscriptstyle\scriptscriptstyle{#1}}%
\!\int}
\def\XXint#1#2#3{{\setbox0=\hbox{$#1{#2#3}{\int}$ }
\vcenter{\hbox{$#2#3$ }}\kern-.6\wd0}}
\def\mint{\Xint-}

\newtheorem{example}{Example}
 \newtheorem{assumption}{Assumption}
\newtheorem{remark}{Remark}
\newtheorem{prop}{Proposition}
\newtheorem{thm}{Theorem}
\newtheorem{lemma}{Lemma}
\newtheorem{definition}{Definition}
\newtheorem{corollary}{Corollary}
\numberwithin{definition}{section}
\numberwithin{thm}{section}
\numberwithin{remark}{section}
\numberwithin{prop}{section}
\numberwithin{corollary}{section}
\numberwithin{assumption}{section}
\numberwithin{lemma}{section}
\newtheorem{thmpf}{Theorem Statement}
\newtheorem{proppf}{Proposition Statement}

\newcommand{\tbeta}{\tilde \beta}
\newcommand{\oJ}{\overline{J}}
\newcommand{\oI}{\overline{I}}
\newcommand{\oP}{\overline{P}}
\newcommand{\wH}{\widehat{H}}
\newcommand{\Id}{I}
\newcommand{\R}{\mathcal{R}}
\newcommand{\I}{\mathcal{I}}
\newcommand{\Tr}{\text{Tr}}
\newcommand{\TrN}{\text{Tr}_N}
\newcommand{\adj}{\text{adj}}
\newcommand{\per}{\text{per}}
\newcommand{\C}{\mathbb{C}}
\newcommand{\J}{\mathcal{J}}
\newcommand{\interior}{\text{int}}
\newcommand{\sch}{\mathcal{S}(\mathbb{R})}
\newcommand{\supp}{\text{supp}}
\newcommand{\Err}{\text{Err}}
\newcommand{\Imag}{\text{Im}}
\newcommand{\Real}{\text{Re}}
\newcommand{\rins}{r_{\text{ins}}}
\newcommand{\hQ}{Q}
\newcommand{\E}{\mathcal{E}}
\newcommand{\Mat}{\tilde H}
\newcommand{\G}{\mathcal{G}}
\newcommand{\Gt}{\widetilde{\mathcal{G}}}
\newcommand{\dhh}{\widehat{\delta h}}
\newcommand{\Z}{\mathbb{Z}^2}
\newcommand{\K}{\mathcal{R}^*} %{\mathcal{K}}\
\newcommand{\tK}{{\tilde \R^*}}
\newcommand{\Br}{B_r(0)}
\newcommand{\B}{\mathcal{B}}
\newcommand{\A}{\mathcal{A}}
\newcommand{\yt}{\tilde{y}}
\newcommand{\Rs}{\mathscr{R}}
\newcommand{\Ha}{\mathcal{H}}
\newcommand{\SN}{\mathcal{S}_N}
\newcommand{\gap}{\text{gap}}
\newcommand{\inter}{\text{inter}}
\newcommand{\intra}{\text{intra}}
\newcommand{\Gammat}{\tilde{\Gamma}}
\newcommand{\OmegaMon}{\Omega}
\newcommand{\D}{\mathcal{D}}
\newcommand{\Hmon}{H}
\newcommand{\MatSpace}{ M_{|\Omega_r|}(\mathbb{C})}
\newcommand{\M}{\mathcal{M}}
\newcommand{\Msup}{S[\widehat{H}]}
\newcommand{\nmod}{\text{mod}}
\newcommand{\hOmega}{\widehat{\Omega}}
\newcommand{\T}{\mathcal{T}}
\newcommand{\U}{\mathcal{U}}
\newcommand{\rc}{r_c}
\newcommand{\V}{\mathcal{V}}
\newcommand{\moire}{\theta}
\newcommand{\mP}{\mathcal{P}}
\newcommand{\LL}{\mathcal{L}}
\newcommand{\bR}{\mathbf{R}^*}
\newcommand{\btR}{\mathbf{\tilde{R}^*}}
\newcommand{\interstrength}{V}
\newcommand{\X}{\mathcal{X}}
\newcommand{\Hr}{\widehat{H}_{\text{sc}}}
\newcommand{\Hstart}{\mathcal{H}_{\text{sc}}}
\newcommand{\op}{\text{op}}
\newcommand{\Kreg}{K_{\text{dom}}}
\newcommand{\TrLim}{\underline{\Tr}}
\newcommand{\Breal}{\mathfrak{B}(\ell^1(\Omega))}
\newcommand{\Bmom}{\mathfrak{B}(\ell^1(\Omega^*))}
\newcommand{\Hm}{\mathcal{H}}
\newcommand{\F}{\mathcal{F}}
\newcommand{\genF}{ \mathcal{T}}
\newcommand{\bloch}{\mathcal{B}_\Omega}
\newcommand{\hu}{\mathfrak{h}}
\newcommand{\Si}{S}
\newcommand{\contourprod}{\mathcal{S}}
\newcommand{\Contour}{\mathfrak{C}}
\newcommand{\energy}{\Sigma}
\newcommand{\mon}{\mathfrak{m}}
\newcommand{\dof}{\Omega^*}

\newcommand{\config}{\mathcal{X}}
\newcommand{\hopInter}{\mathcal{S}}
\newcommand{\OpConfig}{\mathcal{O}}
\newcommand{\Underlying}{\mathcal{U}}
\newcommand{\Gen}{\mathcal{\pi}}
\newcommand{\recip}{\mathcal{E}} % reciprocal space, R[H] operator over S_\Omega
\newcommand{\blochmap}{\mathcal{U}}
\newcommand{\realspace}{\ell^2(\Omega)}
\newcommand{\tr}{\text{tr}}
\newcommand{\Analytic}{\mathcal{H}}

\newcommand{\InterStrength}{E_{\text{inter}}}
\newcommand{\rinteract}{{\tau}}

\newcommand{\erg}{\text{erg}}
\newcommand{\rl}{\text{rl}}
\newcommand{\cf}{\text{cf}}
\newcommand{\rp}{\text{rp}}
\newcommand{\ms}{\text{ms}}
\newcommand{\hopspace}{\UnderlyingSpace}
\newcommand{\name}{\text{arb}}
\newcommand{\eps}{\varepsilon}
\newcommand{\UnderlyingSpace}{\mathbb{H}}
\newcommand{\herm}{\text{herm}}

\section{ Introduction }
\label{sec:intro}

Interest in accurate models for twisted incommensurate materials has exploded in recent years after the discovery of superconductivity in twisted bilayer graphene at the so-called magic angle~\cite{Cao2018sc}. 
2D materials with almost identical periodicities form large scale moir\'e patterns ~\cite{hofstadterkim,incommensurategeim,Dai2016TwistedBG} that are generally {\em incommensurate}~\cite{chen2020plane,espanol2d,massatt2017,Prodan18}, or aperiodic, which has motivated the development of methods to overcome the theoretical and computational challenges posed by the lack of periodicity.  Most current physics investigations overcome the lack of periodicity by utilizing a low-energy continuum approximation that safely removes the details of the precise atomic structure. The most well-known such model is by Bistritzer and MacDonald (BM model)~\cite{bistritzer2011,Catarina2019TwistedBG}, which made a number of  assumptions that greatly simplify the study of twisted bilayer graphene (TBG) near $1^\circ$ twist. Although the BM model is built specifically for TBG, the general framework is applicable for some other materials. The BM model's simple structure and formalism have made it a centerpiece of theoretical work on moir\'e materials. However, its strict assumptions of atomic rigidity and smooth interlayer tunneling \refchange{lead to low accuracy} at twist angles below $1^\circ$~\cite{Carr2019exact}.

Recent work has developed theory and efficient computational methods for studying the electronic structure of incommensurate 2D heterostructures via configuration space and momentum space representations \cite{cances2016,massatt2017,momentumspace17,carr2017,fastkubo19}. These approaches are strongly related, as the BM model can be understood as a momentum space model for TBG with well-chosen approximations simplifying the structure~\cite{wavepacketbm22}. Both the BM model and the momentum space model share computational speedup and physically useful momenta information.  The BM model approximations are in the mechanically unrelaxed regime, and it is known that the mechanical relaxation significantly impacts the geometry and the electronic structure \cite{relaxphysics18,cazeaux2018energy,KimRelax18,Nam2017}. Both relaxation and electronic structure models in \cite{carr2017,relaxphysics18,KimRelax18} are derived from density functional theory (DFT) calculations, so in principle their accuracy is on the level of \dm{Kohn-Sham DFT for these specific systems.}

% In \cite{momentumspace17, massatt2020efficient}, momentum space methods proved to be very powerful for materials nearly aligned with the correct monolayer band structure properties. Graphene and several TMDCs  for example have applicable band structure.  Momentum space methods have the advantage of being asymptotically faster at computing electronic observables such as density of states (DoS) and conductivity \cite{momentumspace17, massatt2020efficient}, and further allow an approximate band structure representation. It suffers the disadvantage of being applicable in fewer settings than configuration representations.

In this work, we consider a generalized class of tight-binding Hamiltonians that allows for mechanical relaxation and general material types, and we prove this class of Hamiltonians can be transformed into a momentum space model. In particular, we start with a formula for the understood real space observable, and we develop \dm{a method} for computing the same observable in the momentum space framework. We present this by building a diagram of \dm{isomorphic} mappings of Hamiltonians over the four relevant spaces for this model: real, configuration, momentum, and reciprocal spaces.  

\dm{The tight-binding model starts with a discrete space $\Omega$ of degrees of freedom, and a collection of hopping functions $\hu$. A real space operator acting on $\ell^2(\Omega)$ is constructed from the hopping functions, denoted $\Gen^\rl(\hu)$. There is a unitary transformation $\G$ that we prove maps this to a momentum framework, a description of the Hamiltonian as a coupling of waves to other waves given by $\Gen^\ms(\tilde \hu) = \G \Gen^\rp(\hu)\G^*$ where $\tilde \hu$ are hopping functions now on the reciprocal lattices describing a momentum ``tight-binding" model. All the interactions between the two 2D materials arise from weak Van der Waals forces, and as a consequence all Hamiltonian terms arising from these interactions will in some sense be perturbations of the isolated 2D material Hamiltonians.

\refchange{Our main result, Theorem \ref{thm:convergence},  gives a family of finite matrices $H(q)$ for momenta $q$ that can be used to approximate observables that are dependent on a finite region of spectra $\energy \subset \mathbb{R}$ for an appropriate class of 2D materials. Theorem \ref{thm:convergence} gives exponential rates of convergence with respect to the hopping truncation and the momenta truncation.
%, and (for the density of states) the Gaussian decay rates. 
The Theorem also proves that the decay rate with respect to the hopping truncation is independent of $\theta$ for the unrelaxed Hamiltonian, but is proportional to $1/\theta$ for the relaxed Hamiltonian. 

Further, $H(q)$ gives a pseudo-band structure as we can write the eigenvalues of $H(q)$ as a function of momenta $q$. It is denoted in the literature by ``quasi-band structure" as the eigenvalues of $H(q)$ do not actually represent precise continuous spectrum, but rather the spectra of $H(q)$ give rise to approximate observables of the true system.  
}

Momentum space models require detailed analysis to construct, and we hope this work will be a bridge to simplify the construction of these models for varying materials and additional mechanical effects such as relaxation by presenting a form for $\hu$, deriving $\tilde \hu$, from $\tilde \hu$ deriving $H(q)$, and proving observables converge exponentially in truncation with only a logarithmic dependence on spectral resolution required for the observable.}
We  note that the classification of the four spaces also gives a strong mathematical foundation for the duality between momentum and configuration space \cite{duality20}, and provides a general class of observables including density of states and the Kubo formula for electronic transport \cite{fastkubo19,massatt2020efficient,cances2016,prodan2012}.

Our second result is the implementation of the momentum space algorithm to TBG~\cite{Nam2017,relaxphysics18,cazeaux2018energy}, along with analysis of the band structure via the momentum space Hamiltonian.  We derive an exact momentum space formulation directly from the real space model without any (uncontrolled) approximations. Numerical tests of the convergence rate of the momentum space algorithm shows stark differences between the unrelaxed and mechanically relaxed atomic geometries. Importantly, we note that a number of the interlayer tunneling approximations that are central to the simple continuum model~\cite{bistritzer2011} no longer hold as the twist angle approaches zero. This has implications for recent attempts to connect realistic models of TBG to the so-called chiral symmetric limit~\cite{Tarnopolsky2019,flat21,Khalaf2019,Wang2021}, which is an analytically solvable version of the BM model which requires the interlayer AA and BB orbital tunnelings (tunneling between orbitals of similar honeycomb sublattice index) to be set to zero. 

The chiral symmetric model of TBG also admits an analytically solvable form of the correlated ground state~\cite{Bultinck2020}, making it an important model Hamiltonian for understanding moir\'e correlated insulators and superconductors. Due to the relaxation of the moir\'e interface into large domains of AB stacking, the effective AA and BB tunneling strengths go to zero proportionally with the twist angle~\cite{Carr2019exact}. Therefore, one may hope that small-angle TBG represents an experimentally achievable form of the chiral symmetric model. But in this strongly relaxed limit we find that one can no longer omit the higher momentum scatterings of the AB and BA tunneling types, as the range of relevant scattering distances grows like the inverse of the twist angle. This prevents the low-angle limit of relaxed TBG from mapping onto the chiral symmetric model, as it includes interlayer scattering from only the three lowest momentum scattering modes.

% In \cite{fang2019angledependent}, the existence of a momentum space formulation is assumed and calculated by numerics.

% We note that the momentum space and configuration space formulations are written as operators over continuous $L^2$ spaces with specified hopping function decay conditions, as opposed to the $C^*$-algebra formulation in \cite{prodan2012,cances2016}. We note that our hopping parameter condition restricts us to operator spaces that are not closed, and thus are not technically $C^*$-algebras. We emphasize the Hilbert space setting here as the momentum transform is unitary with respect to discrete $\ell^2$ to continuous $L^2$, which simplifies the representation and analysis. We note that this framework extends to the multilayer setting, though there are much deeper convergence issues when moving beyond two layers \cite{zhu2020,paul2020,Moon2021}. The approach also sees use in the study of twisted quasicrystals \cite{Moon2019, Crosse2021}, which exhibit an ordered aperiodicity different from the incommensurate structures near $0^\circ$ twist angle.

In Section \ref{sec:generalization}, we present the real, configuration, momentum, and reciprocal spaces and the natural transformations between the four spaces along with the observable formulas. We highlight the relation between the hopping functions in these spaces to show how to move from a real space model to a momentum space model. As discussion of the four spaces involves a fair quantity of notation, we simplify the  representation by keeping almost all notation in Section \ref{subsec:notation} for easy reference. In Section \ref{sec:relax}, we introduce the tight-binding model with mechanical relaxation \dm{and write it in a form compatible with momentum space. In Section \ref{sec:algorithm}, we formulate an efficient algorithm in momentum space, provide a bound for the convergence, and quantify the convergence slow-down from mechanical relaxation effects.} In Section \ref{sec:numerics}, we numerically illustrate the algorithm for twisted bilayer graphene with mechanical relaxation. In Section \ref{sec:proofs}, we put the majority of the proofs, and in Appendix \ref{app:relax} we discuss the mechanical relaxation model, which is considered an input model for the tight-binding Hamiltonian.
%
%In Section \ref{sec:config}, we will review the real space configuration representation, and then review the mechanical and electronic models in real space configurations. We shall then discuss in Section \ref{sec:ms} the transformation from real space to momentum space, and finally transform the relaxed system into momentum space. We shall than discuss how the relaxation tempers the convergence rate of momentum space methods, though it still proves to be significantly faster than real space configuration methods for materials with the appropriate electronic structure. Indeed, we also shall discuss the duality configuration and momentum space, and how the relaxation affects convergence of both methods simultaneously. Both configuration and momentum space methodologies involve integrals over a diagonal element of a matrix function. We will discuss how regularity of the integral in configuration space relates to convergence in the matrix size in the momentum space method, and how both grow in complexity for large moir\'e patterns.

\section{Real, Configuration, Momentum, and Reciprocal Spaces}
\label{sec:generalization}

In this section, we build the four spaces and the isomorphic diagram between them.
Our assumed starting point is a real space tight-binding model defined over two incommensurate lattices with a finite number of orbitals associated to each lattice site forming a discrete basis.  To define the geometry of these lattices in the 2D plane, we write for $j \in \{1,2\}$

\dm{\begin{align}
& \R_j = A_j \Z, & \R_j^* = 2\pi A_j^{-T} \Z, \\
& \Gamma_j = A_j[0,1)^2, & \Gamma_j^* = 2\pi A_j^{-T} [0,1)^2, \\
& \A_1, \A_2 \text{ are finite orbital sets,} & \Gamma_j,\Gamma_j^* \text{ used as tori}.
\end{align}}
$\R_j$ are the real space lattices, $\R_j^*$ are the reciprocal lattices, and $\Gamma_j$, $\Gamma_j^*$ are the corresponding unit cells.
In the tight-binding approximation, associated with each lattice site $R \in \R_j$ there is a set of orbitals $\A_j$. $R\alpha$ then parametrizes all orbitals in the system, where $R \in \R_j$, $\alpha \in \A_j$, $j \in \{1,2\}$. A Hamiltonian operator $H$ then couples orbitals with hopping terms denoted $H_{R\alpha,R'\alpha'}$. These hopping terms are calculated via matrix-valued functions $h$ such that $H_{R\alpha,R'\alpha'} = h_{\alpha\alpha'}(R-R')$. Here $\alpha,\alpha'$ sample the matrix entries of $h(R-R')$. We see the hopping functions define the Hamiltonian. For this reason, we focus on careful book-keeping of the hopping functions through the transformation between spaces.

Consider a quantum wave function in real space restricted to the first sheet, denoted $\psi = \{\psi_{R\alpha}\}_{R\alpha \in \R_1\times \A_1}$. In moir\'e systems, it is useful to label each site $R \in \R_1$ by its respective position to the second sheet.
This disregistry, notated as $b \in \Gamma_2$, is obtained by modulating $R$ with respect to $\Gamma_2$. The indexing of atomic sites by configuration instead of real space location is the basis for configuration space (see Figure \ref{fig:config}). We can then interpret the Hamiltonian $H$ as coupling between sites on the tori $\Gamma_1$ to sites on $\Gamma_2$. The ``hopping'' between different sites will be defined through translation operators over the tori. Momentum space exploits the Bloch basis, which correspond to waves on a single layer parametrized by wavenumber $q \in \Gamma_1^*$, $\psi_{\alpha'}(q) = \{ e^{iq \cdot R} \delta_{\alpha\alpha'}\}_{R\alpha}.$ The presence of a lattice mismatch between the sheets (e.g. a twist) introduces a non-trivial scattering condition between the Bloch bases of the two sheets, coupling a wavenumber $q \in \Gamma_1^*$ to a set of Bloch states with distinct wavenumbers in $\Gamma_2^*$. These wavenumbers in turn couple back to a collection of wavenumbers in sheet one, and so forth. This scattering leads to a lattice model over momenta, which is the basis of momentum space and reciprocal space. This scattering will be understood via translation operators over the reciprocal lattice unit cells $\Gamma_1^*$ and $\Gamma_2^*$ connecting corresponding momenta.

Next, we introduce a compact notation section for easy reference. First, we introduce the four spaces of relevance. Secondly, we discuss the so-called hopping functions, which describe how lattice sites couple. Then we define the Hamiltonian from the hopping functions for the four spaces. Next we define the relevant operator spaces, and finally we introduce the transformations that map between the four spaces.

\begin{figure}[ht]
\centering
\includegraphics[width=.6\textwidth]{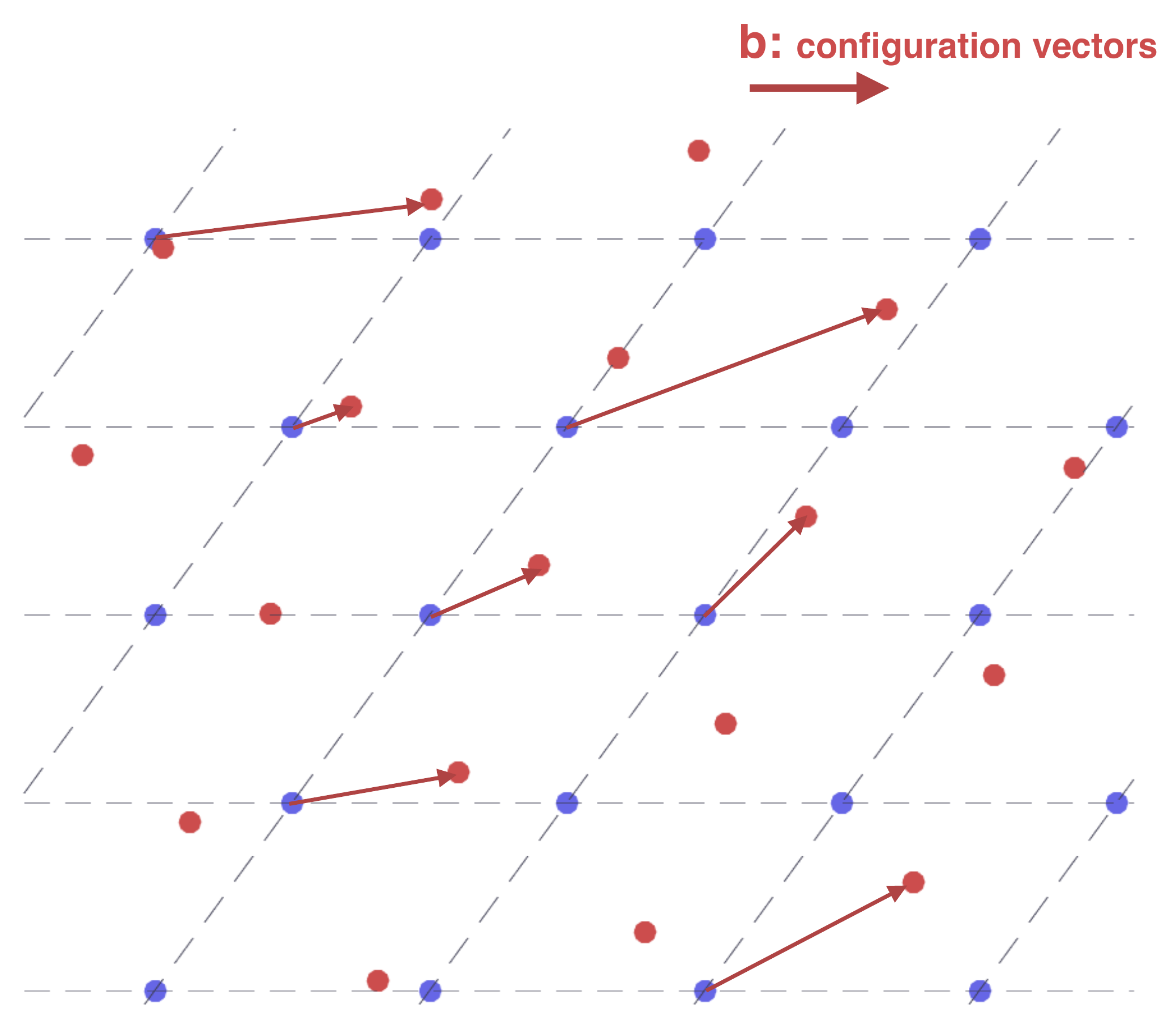}
\caption{A twisted lattice of two sheets, with layer \dm{ two the dots connected by dashed lines, while layer one is the isolated dots. The vectors $b \in \Gamma_2$ (denoted by arrows) parameterize different sites $R \in \R_1$.}}
\label{fig:config}
\end{figure}

\subsection{ Notation }
\label{subsec:notation}

\begin{figure}[ht]
\includegraphics[width=.9\textwidth]{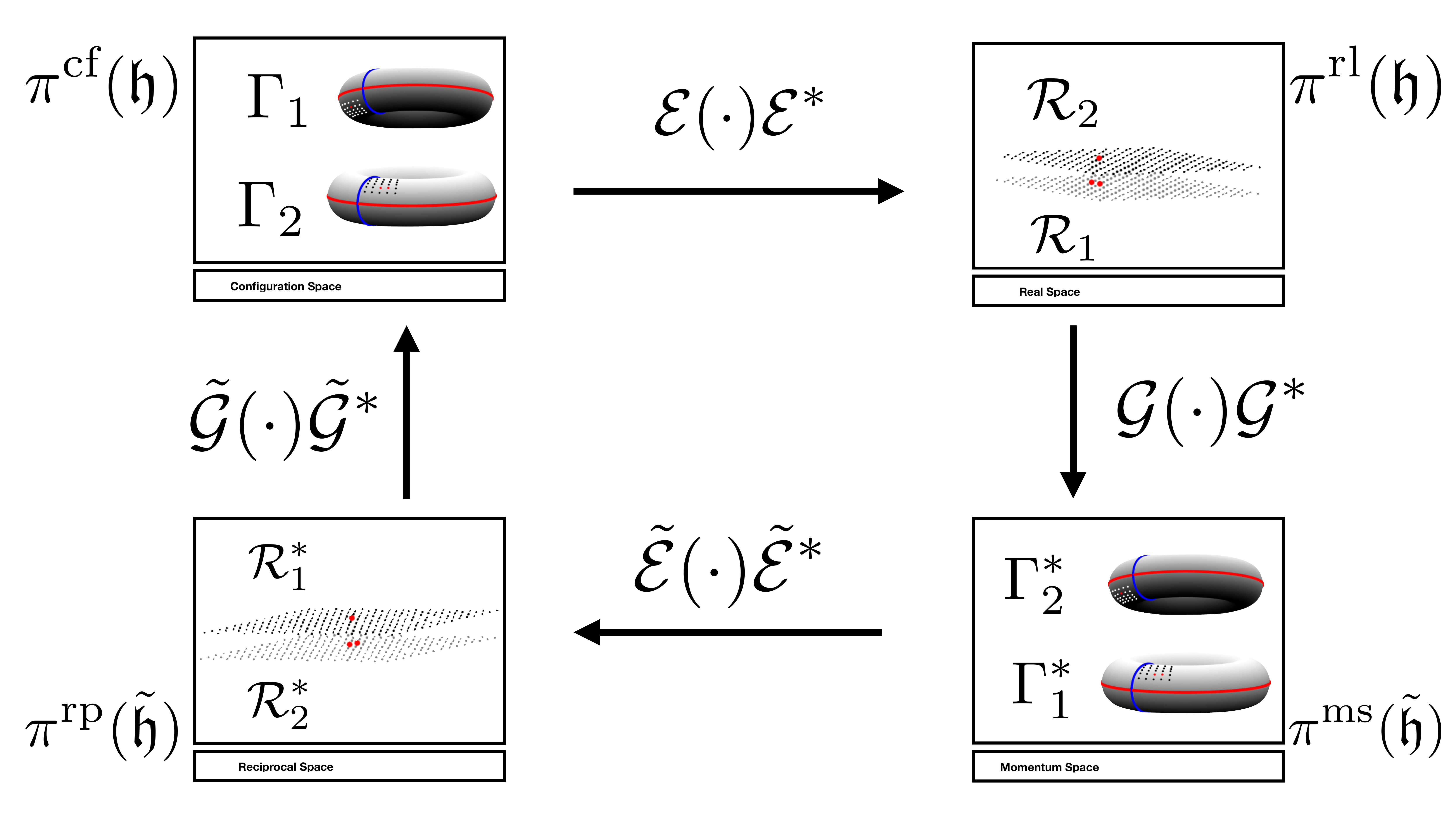}
\caption{The isomorphic diagram is presented above with configuration space on top-left, real space on top-right, momentum space bottom-right, and reciprocal space bottom-left. The rectangular image gives a pictorial representation of the space, the operators are listed on the corners, and the transformations of the operators  are written over the arrows between the spaces. \refchange{We note that there are two Hilbert spaces and inner products associated with real and reciprocal space, so the diagram, while isomorphic, isn't unitary.} }
\label{figure:commutative}
\end{figure}

We will use $G$ to represent entries of reciprocal lattices, $R$ for real space lattice entries, and $\alpha$ for orbitals. We will typically skip reiterating which lattice or orbital set they are in, as the operator and function space context will make this apparent.

\subsubsection{Four spaces}

`rl' will be used to denote real space, `cf' configuration space, `rp' reciprocal space, and `ms' momentum space. First, we define the four spaces and their sheet decompositions. $\X$ will be used to denote the spaces. A subscript of $1$ or $2$ will be the space restricted to sheet $1$ or $2$ respectively, and the superscript will denote the space, either `rl', `cf', `rp', or `ms.'
\begin{align*}
& \Omega_1 = \R_1 \times \A_1, & \Omega_2 = \R_2\times \A_2, & \hspace{1cm} \Omega = \Omega_1 \cup \Omega_2,\\
&\Omega_1^* = \R_2^* \times \A_1, & \Omega_2^* = \R_1^* \times \A_2 , & \hspace{1cm}  \Omega^* = \Omega_1^*\cup\Omega_2^*,\\
& \X_1^\cf = L^2_\per(\Gamma_2;\C^{\A_1}), &\X_2^\cf = L^2_\per(\Gamma_1;\C^{\A_2}), & \hspace{1cm} \X^\cf = \X_1^\cf \oplus \X_2^\cf, \\
& \X_1^\ms = L^2_\per(\Gamma_1^*; \C^{\A_1}), & \X_2^\ms = L^2_\per(\Gamma_2^*; \C^{\A_2}), &  \hspace{1cm}\X^\ms = \X_1^\ms \oplus \X_2^\ms, \\
& \X_1^\rl = \ell^2(\Omega_1), & \X_2^\rl = \ell^2(\Omega_2), & \hspace{1cm} \X^\rl = \X_1^\rl \oplus \X_2^\rl =  \ell^2(\Omega), \\
& \X_1^\rp = \ell^2(\Omega_1^*), & \X_2^\rp = \ell^2(\Omega_2^*), &\hspace{1cm}  \X^\rp = \X_1^\rp \oplus \X_2^\rp = \ell^2(\Omega^*).
\end{align*}
When we use $\psi \in \X^\name$ where `$\name$' is either `rl', `ms', `rp', or `cf', then we will denote the decomposition into sheets as $\psi = (\psi_1,\psi_2)^T$ for $\psi_j \in \X_j^\name$, $\psi \in \X^\name$.

\subsubsection{Hopping functions}
Before defining the hopping functions, we define a couple of relevant spaces. We let $M_{ij}$ be the space of complex-valued $|\A_i|\times |\A_j|$ matrices.
We denote $ \Hm(\mathbb{T}_1,\mathbb{T}_2;M)$ to be the space of multi-variable analytic functions over the tori $\mathbb{T}_j$ whose elements $h$ have corresponding Fourier modes $h_{L_1L_2} \in M$ where $L_j$ is a lattice vector of the Bravais lattice with corresponding unit cell $\mathbb{T}_j,$ i.e.,
\refchange{\begin{gather*}
h_{L_1L_2}=\frac1{|\mathbb{T}_1|\,|\mathbb{T}_2|}\int_{\mathbb{T}_1}\int_{\mathbb{T}_2}h(\xi_1,\xi_2)e^{-i(L_1\cdot\xi_1+L_2\cdot\xi_2)}d\xi_1\,d\xi_2,\\
h(\xi_1,\xi_2)=\sum_{L_1}\sum_{L_2}h_{L_1L_2}e^{i(L_1\cdot\xi_1+L_2\cdot\xi_2)}.
\end{gather*}}
Here $M$ is some vector space, for example the $M_{ij}$'s. \dm{The intralayer hopping of sheet one for configuration and momentum space are respectively denoted as
\begin{align}
&h_R(b) := \sum_{G \in \R_2^*} h_{RG}e^{iG\cdot b}
\refchange{=\frac{1}{\Gamma_1^*} 
\int_{\Gamma^*_1}h(q,b)e^{-iR\cdot q}dq},
&h \in \Hm(\Gamma_1^*,\Gamma_2; M_{11}), \; R \in \R_1, \\
&h_G(q) := \sum_{R \in \R_1} h_{GR}e^{i R \cdot q}\refchange{=\frac{1}{\Gamma_2} 
\int_{\Gamma_2}h(b,q)e^{-iG\cdot b}db}, & h \in \Hm(\Gamma_2, \Gamma_1^*; M_{11}), \; G \in \R^*_2.
\end{align}}
Parallel notation will be used for sheet 2 intralayer hopping.

We use the following convention for the Fourier transform and its inverse:
\begin{align}
&\hat h(\xi) =\frac{1}{(2\pi)^2}\int h(x)e^{-ix\cdot \xi}dx,& \check{h}(x) =    \int h(\xi) e^{ix\cdot \xi}d\xi.
\end{align}
To help define interlayer hopping functions, we define the space
\[
\hopInter(M):= \{ h \in L^2(\mathbb{R}^2;M) : |h(x)| \lesssim e^{-\gamma |x|}, |\hat h(\xi)| \lesssim e^{-\gamma' |\xi|} \text{ for some } \gamma,\,\gamma' > 0\}.
\]
Note that intralayer hopping functions have Fourier modes that decay exponentially due to the analyticity, and interlayer coupling also exhibits exponential decay by definition. This decay rate is useful for Combes-Thomas type estimates of the resolvents. \dm{The model is reasonable as many tight-binding models are approximated by Wannier orbitals, which either exhibit exponential decay or can be reasonably approximated by exponential decay \cite{wannier2007, shiang2016}.}

Next we describe momentum space and configuration space hopping functions. One set of coupling functions will uniquely determine the Hamiltonian in all four spaces
\begin{align*}
&\hopspace^\cf =  \begin{pmatrix} \Analytic(\Gamma_1^*,\Gamma_2;M_{11}) & \hopInter(M_{12}) \\ \hopInter(M_{21}) & \Analytic(\Gamma_2^*,\Gamma_1;M_{22})\end{pmatrix},\\
&\hopspace^\ms =  \begin{pmatrix} \Analytic(\Gamma_2,\Gamma_1^*;M_{11}) & \hopInter(M_{12}) \\ \hopInter(M_{21}) & \Analytic(\Gamma_1,\Gamma_2^*;M_{22})\end{pmatrix} ,\\
&\hopspace^\cf_{\herm} = \biggl\{ \hu = \begin{pmatrix} \hu_{11} & \hu_{12} \\ \hu_{21} & \hu_{22}\end{pmatrix} \in \hopspace^\cf \, : \, [\hu_{jj}]_{RG} = [ \hu_{jj}]_{-R,-G}^*e^{-iG\cdot R} ,\hspace{2mm} \dm{\hu_{12}(b) = \hu_{21}^*(-b)} \biggr\} ,\\
&\hopspace^\ms_\herm = \biggl\{ \tilde \hu = \begin{pmatrix} \tilde\hu_{11} & \tilde\hu_{12} \\\tilde \hu_{21} & \tilde\hu_{22}\end{pmatrix} \in \hopspace^\ms \, : \,[\tilde\hu_{jj}]_{GR} = [\tilde \hu_{jj}]_{-G,-R}^*e^{-iG\cdot R} ,\hspace{2mm} \tilde \hu_{12} = \tilde\hu_{21}^* \biggr\}.
\end{align*}
Here we use `herm' to refer to hermitian, and in the setting of hopping functions we should understand it as hopping functions that give rise to Hermitian or self-adjoint operators.

 We recognize there are multiple subscripts necessary for the notation, and as such we use brackets to separate objects and their sheet index labels from the orbital or Fourier mode indices, which will be outside the brackets. For example, if $\hu \in \dm{\UnderlyingSpace^\cf_\herm}$, then $[\hu_{11}]_{RG,\alpha\alpha'}$ is the intralayer sheet 1 hopping function's $(R,G)$ Fourier mode of the $(\alpha,\alpha')$ entry. 
If $\hu \in \hopspace^\cf_\herm$, we will find the equivalent hopping functions for momentum space given by the following $\tilde \hu \in \hopspace_\herm^\ms$.
Let $c_j^* = |\Gamma_j^*|^{1/2}.$ 
\begin{align}
& \dm{\tilde \hu_{jj}(b,q) = \hu_{jj}^*(-q,-b)} , \\
& \tilde \hu_{ij}(\xi) = c_1^*c_2^* \hat\hu_{ij}(\xi).
\end{align}
Here $i \neq j$.

\subsubsection{Operators}

Before defining the operators and Hamiltonians, we define a few symmetry operators:
\dm{
\begin{align*}
& T_R: \X_j^\cf \rightarrow \X_j^\cf, & T_R\;\psi(b)  = \psi(b+R), \\
& T_G: \X_j^\ms \rightarrow \X_j^\ms, & T_G\;\psi(q) = \psi(q+G), \\
&\Si : \X_j^\cf \rightarrow \X_j^\cf, & \Si\;\psi(b) = \psi(-b).
\end{align*}}
Suppose $\hu \in \UnderlyingSpace_\herm^\cf$ with corresponding $\tilde \hu$. Then we define the corresponding generated operators, where matrix-valued functions are understood as multiplication operators. Below we assume $i \neq j$:
\begin{align*}
& \Gen^\cf_{j\leftarrow j}(\hu_{jj}) := \sum_{R\in\R_j}[\hu_{jj}]_R(\cdot)T_{-R}, & \Gen^\cf_{i \leftarrow j}(\hu_{i \leftarrow j}) := \sum_{R \in \R_j} [\hu_{ij}](\cdot-R)ST_{R},\hspace{2mm} \\
& \Gen^\ms_{i\leftarrow i}(\tilde \hu_{ii}) := \sum_{G \in \R_j^*}[\tilde \hu_{ii}]_G(\cdot)T_{-G},  & \Gen^\ms_{i\leftarrow j}(\tilde \hu_{ij}) := \sum_{G \in \R_i^*} \tilde \hu_{ij}(\cdot+G)\dm{T_{G}},\hspace{2mm} \\
& \biggl(\Gen^\rl_{i \leftarrow i}(\hu_{ii})\biggr)_{R\alpha,R'\alpha'} = [\hu_{ii}]_{R-R',\alpha\alpha'}(R), & \biggr(\Gen^\rl_{j \leftarrow i}(\hu_{ji}) \biggr)_{R\alpha,R'\alpha'}= [\hu_{ji}]_{\alpha\alpha'}(R-R'), \\
& \biggl(\Gen^\rp_{i \leftarrow i}(\tilde \hu_{ii})\biggr)_{G\alpha,G'\alpha'} = [\tilde \hu_{ii}]_{G-G',\alpha\alpha'}(G), & \biggr(\Gen^\rp_{j \leftarrow i}(\tilde \hu_{ji}) \biggr)_{G\alpha,G'\alpha'}= [\tilde \hu_{ji}]_{\alpha\alpha'}(G+G'),\\
& \Gen^\cf(\hu) = \begin{pmatrix} \Gen^\cf_{1 \leftarrow 1}(\hu_{11}) & \Gen^\cf_{1 \leftarrow 2}(\hu_{12}) \\ \Gen^\cf_{2\leftarrow 1}(\hu_{21}) & \Gen^\cf_{2\leftarrow 2}(\hu_{22})\end{pmatrix}, & \Gen^\ms(\tilde \hu) = \begin{pmatrix} \Gen^\ms_{1 \leftarrow 1}(\tilde \hu_{11}) & \Gen^\ms_{1 \leftarrow 2}(\tilde \hu_{12}) \\ \Gen^\ms_{2\leftarrow 1}(\tilde \hu_{21}) & \Gen^\ms_{2\leftarrow 2}(\tilde \hu_{22})\end{pmatrix}, \\
&\Gen^\rl(\hu) = \begin{pmatrix} \Gen^\rl_{1 \leftarrow 1}(\hu_{11}) & \Gen^\rl_{1 \leftarrow 2}(\hu_{12}) \\ \Gen^\rl_{2\leftarrow 1}(\hu_{21}) & \Gen^\rl_{2\leftarrow 2}(\hu_{22})\end{pmatrix}, & \Gen^\rp(\tilde \hu) = \begin{pmatrix} \Gen^\rp_{1 \leftarrow 1}(\tilde \hu_{11}) & \Gen^\rp_{1 \leftarrow 2}(\tilde \hu_{12}) \\ \Gen^\rp_{2\leftarrow 1}(\tilde \hu_{21}) & \Gen^\rp_{2\leftarrow 2}(\tilde \hu_{22})\end{pmatrix}.
\end{align*}

\subsubsection{Operator Spaces}

We define the spaces of \dm{Hamiltonian} operators over the four spaces.
\dm{\begin{align}
& \OpConfig^\rl_\herm = \Gen^\rl(\UnderlyingSpace^\cf_\herm), & \OpConfig^\cf_\herm = \Gen^\cf(\UnderlyingSpace^\cf_\herm), \\
& \OpConfig^\rp_\herm = \Gen^\rp(\UnderlyingSpace^\ms_\herm), & \OpConfig^\ms_\herm = \Gen^\ms(\UnderlyingSpace^\ms_\herm).
\end{align}
\dm{We will construct a space of operators associated to observables dependent on a set of Hamiltonian operators. This generalization includes observables such as entries of the Kubo formula, density of states, and Chern numbers. } For $H$ an operator, we define
\begin{align*}
   & \Contour_\varepsilon(H) = \{ C \text{ a contour around } \sigma(H) : d(C, \sigma(H) ) \geq \varepsilon\}, \\
   & \text{Int}(\contourprod) := \times_{j=1}^n \text{Int}(C_j), \\
   & \contourprod := \times_{j=1}^n C_j, \\
   &\Lambda(\contourprod) := \{ g \text{ analytic on } \text{Int}(\contourprod) \},
\end{align*}
where $\text{Int}(C_j)$ denotes the interior of the contour $C_j$ and $d(A,B)$ denotes the distance between the sets $A$ and $B$ in the complex plane. 
We denote $dz = dz_1\cdots dz_n,$ $z = (z_1,\cdots z_n)$, and
\begin{align*}
& \OpConfig_\varepsilon^\name = \biggl\{ \int_\contourprod g(z)\prod_{j=1}^n(z_j-H_j)^{-1}dz : n > 0, \; H_1,\cdots H_n \in \OpConfig^\name_\herm, \; C_j \in \Contour_\varepsilon(H_j),\; g \in \Lambda(\C)\biggr\}, \\
& \OpConfig^\name = \cup_{\varepsilon > 0} \; \OpConfig_\varepsilon^\name.
\end{align*}
We represent $O \in \OpConfig^\name$ by the set $(g,H_1,\cdots H_n)$.
}

\subsubsection{Observables}

Computing observables of quantum systems requires taking traces of appropriate objects. 
For $\tilde \hu \in \UnderlyingSpace^\ms_\herm$, we define $t_{q'}$ by
\begin{align}
&(t_{q'} \tilde \hu)_{jj}(b,q) = \tilde \hu_{jj}(b,q+q'), & (t_{q'}\tilde \hu)_{ij}(\xi) = c_1^*c_2^* \tilde \hu_{ij}(\xi + q').
\end{align}
Likewise we have for \refchange{ $\hu \in \UnderlyingSpace^\cf_\herm$}
\begin{align}
&(t_{b'} \hu)_{jj}(q,b) = \hu_{jj}(q,b+(-1)^{j+1}b'), & (t_{b'}\hu)_{ij}(x) = \hu_{ij}(x + (-1)^ib').
\end{align}
We use $O^\name \in \OpConfig^\name$ to be an operator of an observable.
Reciprocal and real spaces have thermodynamic limit traces, which we denote
\begin{align*}
& \TrLim \; O^\rl = \lim_{r\rightarrow\infty}\frac{1}{\# \Omega_r} \sum_{R\alpha \in \Omega_r} [O^\rl]_{R\alpha,R\alpha},& \Omega_r = \{ R\alpha \in \Omega: |R| < r\}, \\
&\TrLim \; O^\rp =\lim_{r\rightarrow\infty} \frac{1}{\# \Omega_r^*} \sum_{G\alpha \in \Omega_r^*} [O^\rp]_{G\alpha,G\alpha},& \Omega_r^* = \{ G\alpha \in \Omega^*: |G| < r\}.
\end{align*}
{We define traces over configuration space and momentum space as follows:
\begin{equation*}
\Tr \; [O^\cf]_{jj} = \int_{\Gamma_{3-j}}  \tr \;[\hu_{jj}]_0(b)db,\hspace{1cm}
\Tr \; [O^\ms]_{jj} = \int_{\Gamma_j^*} \tr \; [\tilde \hu_{jj}]_0(q)dq.
\end{equation*}
Traces over the full momentum and configuration spaces are given by}
\begin{align*}
& \Tr \; O^\cf = \nu\sum_{j=1}^2  \Tr \; [O^\cf]_{jj},& \nu = (|\A_1|\cdot|\Gamma_2| + |\A_2| \cdot |\Gamma_\refchange{1}|)^{-1}, \\
& \Tr \; O^\ms = \nu^*\sum_{j=1}^2  \Tr \; [O^\ms]_{jj}, & \nu^* = (|\A_1| \cdot |\Gamma_1^*| + |\A_2| \cdot |\Gamma_2^*|)^{-1}.
\end{align*}
 %\dm{Here we define $\Tr$ as the standard $L^2$ trace. For example, $\Tr \; [O^\ms]_{jj}$ is the $L^2$ trace over $\X_j^\ms$.}

\subsubsection{Transformations}

We noted that the \dm{hopping} functions uniquely determined the Hamiltonian in all four spaces. Real and configuration Hamiltonians share the same underlying functions, and likewise momentum and reciprocal. We thus define the natural transform from configuration (momentum) to real (reciprocal). \refchange{We first define two separate inner products and associated Hilbert spaces for real space and reciprocal space. To do this, we first define $$\X_a^\name = \{ \psi \in \X^{\name} \; :\; \psi \text{ is analytic}\}$$ for `$\name$' either `$\cf$' or `$\ms$'. Then without writing the Hilbert space associated with the range quite yet, we write
 $\recip: \config^\cf_a \rightarrow \recip(\config^\cf_a)$ and $\tilde \recip : \config_a^\ms \rightarrow \recip(\config_a^\ms)$ by
\begin{align}
& \recip\psi_{R\alpha} = \psi_\alpha(R), & \alpha \in \A_j,\\
& \tilde\recip\psi_{G\alpha} = \psi_\alpha(G), & \alpha \in \A_j.
\end{align}
Next we define the {\em ergodic inner products}, and the {\em ergodic Hilbert spaces}, associated with the completion of the range of $\recip_b$ and $\tilde\recip_q$. Consider arbitrary $\psi,\phi \in \X^\cf_a$. Let $\# \Omega_r$ correspond to the cardinality of $\Omega_r$.  Then\cite{ cances2016,massatt2017}
\begin{equation*}
    \begin{split}
        \langle \recip\psi, \recip\phi\rangle_\erg &:= \lim_{r\rightarrow \infty} \frac{1}{\# \Omega_r} \sum_{R\alpha \in \Omega_r} \langle \recip\psi_{R\alpha}, \recip\phi_{R\alpha}\rangle  \\
        &= \nu\biggl(\sum_{\alpha \in \A_1}\int_{\Gamma_2} \langle \psi_1(b),\phi_1(b)\rangle db + \sum_{\alpha \in \A_2} \int_{\Gamma_1} \langle \psi_2(b), \phi_2(b)\rangle db\biggr) \\
        &= \langle \psi, \phi\rangle.
    \end{split}
\end{equation*}
We call the associated Hilbert space $\X^\rl_a$.
For $\psi,\phi \in \X^\ms_a$, we define the inner product
\begin{equation*}
    \begin{split}
        \langle \tilde\recip\psi, \tilde\recip\phi\rangle_\erg &:= \lim_{r\rightarrow \infty} \frac{1}{\# \Omega^*_r} \sum_{G\alpha \in \Omega^*_r} \langle \tilde\recip\psi_{G\alpha}, \tilde\recip\phi_{G\alpha}\rangle  \\
        &= \nu^*\biggl(\sum_{\alpha \in \A_1}\int_{\Gamma_1^*} \langle \psi_1(q),\phi_1(q)\rangle dq + \sum_{\alpha \in \A_2} \int_{\Gamma_2^*} \langle \psi_2(1), \phi_2(q)\rangle dq\biggr) \\
        &= \langle \psi, \phi\rangle.
    \end{split}
\end{equation*}
Here 
\begin{align*}
    &\nu = (|\A_1|\cdot|\Gamma_2| + |\A_2|\cdot |\Gamma_1|)^{-1}, & \nu^* = (|\A_1|\cdot|\Gamma_1^*| + |\A_2|\cdot |\Gamma_2^*|)^{-1}.
\end{align*}
We note these operators are unitary mappings between Hilbert spaces.}
We define \refchange{
\begin{align}
    &\blochmap_\recip(O^\cf) = \recip(O^\cf)\recip^*, & O^\cf \in \OpConfig^\cf_\herm,\\
    &\blochmap_{\tilde\recip}(O^\ms) = \tilde\recip(O^\ms)\tilde\recip^*, & O^\ms \in \OpConfig^\ms_\herm.
\end{align}}
Here $\recip$ and $\tilde \recip$ are utilizing the ergodic structure of configuration and momentum space to unfold via ergodicity onto infinite incommensurate lattices \refchange{as described above. We note that $\OpConfig_\rl(\hu)$ can be seen either over the Hilbert space $\X_a^\rl$ or $\X^\rl$ (and likewise for reciprocal space). Since the representation of the operator is the same, we use the same notation for both. It will be important to realize however that in the diagram in Figure \ref{figure:commutative}, the operations $\blochmap_\recip$ and $\blochmap_{\tilde \recip}$ map operators over configuration and momentum space to operators over the completion of $\X^\rl_a$ and $\X^\rp_a$ respectively. The momentum transformations we next define will now consider the same operators $\OpConfig^\cf(\hu)$ and $\OpConfig^\ms(\hu)$ as operators over $\X^\rl$ and $\X^\rp$ respectively. While this representation switch clearly changes the operators and the Hilbert space, a result of this work is that observables remain the same regardless of this switch in Hilbert spaces.

We also wish the real and reciprocal spaces to be able to describe various local configurations in configuration and momentum space, and to this end we define the shifted ergodic maps
\begin{align}
& \recip_b = \recip \begin{pmatrix} T_b & 0 \\ 0 & T_{-b}\end{pmatrix},\\
& \tilde \recip_q = \recip \begin{pmatrix} T_q & 0 \\ 0 & T_q\end{pmatrix}
\end{align}
acting on configuration and momentum spaces respectively.
}

The transformation from real (reciprocal) to momentum (configuration) is given by the correct Bloch transform. We define $\G_j : \X_j^\rl \rightarrow \X_j^\ms$ and $\tilde \G_j : \X_j^\rp \rightarrow \X_j^\cf$ and combine them to form the unitary transformations $\G : \X^\rl \rightarrow \X^\ms$ and $\tilde \G : \X^\rp \rightarrow \X^\cf$ as follows:
\begin{align}
& \G_1 \psi_\alpha(q_1) =|\Gamma_1^*|^{-1/2} \sum_{R \in \R_1} e^{-iq_1\cdot R} \psi_{R\alpha}, & \tilde \G_1 \psi_\alpha(b_2) = |\Gamma_2|^{-1/2} \sum_{G \in \R_2^*} e^{iG\cdot b_2} \psi_{G\alpha},\\
& \G_2 \psi_\alpha(q_2) = |\Gamma_2^*|^{-1/2}\sum_{R \in \R_2} e^{-iq_2\cdot R} \psi_{R\alpha}, & \tilde \G_2 \psi_\alpha(b_1) = |\Gamma_1|^{-1/2}\sum_{G \in \R_1^*} e^{iG \cdot b_1} \psi_{G\alpha}, \\
& \G \begin{pmatrix} \psi_1(q_1)\\ \psi_2(q_2)\end{pmatrix} = \begin{pmatrix} \G_1 \psi_1(q_1) \\ \G_2 \psi_2(q_2) \end{pmatrix}, &  \tilde \G \begin{pmatrix}\psi_1(b_2) \\ \psi_2(b_1)\end{pmatrix} = \begin{pmatrix} \tilde \G_1 \psi_1(b_2) \\ \tilde \G_2 \psi_2(b_1) \end{pmatrix}.
\end{align}
\dm{The unitary transformations $\blochmap_\G := \G(\cdot)\G^*$ and $\blochmap_{\tilde \G} := \tilde \G(\cdot)\tilde \G^*$ complete the \dm{isomorphic} diagram.} The sign change in the definition of $\G_j$ and $\tilde \G_j$ was chosen for the user's book-keeping convenience.

\subsubsection{Summary}
In summary, for $\name$ either `rl', `cf', `ms', or `rp', $\X^\name$ corresponds to the basis of the quantum wave functions, $\OpConfig^\name$ to the space of operators, $\OpConfig^\name_\herm$ to the self-adjoint Hamiltonian operators, and $\Gen^\name$ to the transformation of hopping functions into the operator space. For $\name$ either `ms' or `cf',  $\hopspace^\name_\herm$ is the space of hopping functions that give rise to self-adjoint Hamiltonian operators. $\blochmap_\G$, $\blochmap_{\tilde \G}$, \dm{$\blochmap_\recip$, and $\blochmap_{\tilde\recip}$} form the \dm{isomorphic} diagram between the four operator spaces \dm{(see Figure} \ref{figure:commutative}). A key take away here is that the formal notation gives direct formulas for computing the hopping functions corresponding to momentum and reciprocal spaces from the hopping functions corresponding to real and configuration spaces, which will allow population of the matrices necessary for simulations over momentum space.

% \subsection{Algorithm Notation}
% If $\tilde \hu \in \UnderlyingSpace^\ms$, we denote the momentum shift $t_{q'}(\tilde \hu) \in \UnderlyingSpace^\ms$ such that 
% \begin{align}
% &t_{q'}(\tilde \hu)_{jj}(b,q) =\tilde \hu_{jj}(b,q+q'), &t_{q'}(\tilde \hu)_{ij}(\xi) = \tilde \hu_{ij}(q'+\xi).
% \end{align}
% $\Gen^\ms(\tilde \hu)$ can be seen as a fiber bundle of `infinite matrix' operators across the tori $\Gamma_j^*$, in particular the operators $\Gen^\rp(t_q\tilde \hu)$, $q \in \Gamma_j^*$. Algorithmically this can be reduced to a bundle over the `incommensurate Brillouin zone' $\Gamma_{12}^* := 2\pi(A_2^{-T}-A_1^{-T})[0,1)^2$. 

% \begin{figure}[ht]
% \includegraphics[width=1\textwidth]{figures/Algorithm}
% \end{figure}

% RELAXATION FIGURE
% ==========================================

% \begin{figure}[ht]
% \centering
% \includegraphics[width=.6\textwidth]{figures/relaxation.png}
% \caption{ Here we show atomistic positions in bilayer MoS2 with a relative twist of $1.7^\circ$. As can be clearly seen, the atomistic positions change significantly compared to their original unrelaxed positions. Naturally we anticipate significant changes to the electronic profile as a result.
% Displayed is twisted bilayer MoS$_2$ at 1.7$^\circ$ degrees, but with exageratted relaxation coefficients (from a 0.5$^\circ$  calculation), for illustrative purposes.}
% \label{fig:relaxation}
% \end{figure}

% ==========================================

\subsection{The Isomorphic Diagram of the Four Spaces}
\label{sec:commutative}

As outlined in the notation section, we have four spaces in which to represent the Hamiltonian through the hopping functions. $\X^\ms$ and $\X^\cf$ are continuous $L^2$ spaces that act over momenta and local configurations respectively, while $\X^\rp$ and $\X^\rl$ are discrete $\ell^2$ lattice models \refchange{We also have the ergodic versions of reciprocal and real space $\X^\rp_a$ and $\X^\rl_a$. When working with the ergodic transformations $\recip$ or $\tilde \recip$, we will assume the ergodic inner products and Hilbert spaces for reciprocal and real spaces. When working with the Bloch transforms, we assume the $\ell^2$ inner products and Hilbert spaces.}
 The objective of this section is to derive the relations between the four spaces. Most importantly for numerics, this will give a clear connection between the formulas for observable calculations in real space to that of momentum space, the latter space being the space where the BM model arises from.

\dm{We begin by stating the transformation of operators from configuration to real space, and from momentum to reciprocal space, which involves ergodic unfolding.
\begin{thm}
\label{thm:ergodicunfolding}
For $\hu \in \UnderlyingSpace^\cf_\herm$, we have
\begin{align}
&\blochmap_{\recip_b} \bigl(\Gen^\cf(\hu)\bigr) = \Gen^\rl(t_b\hu), \\
&\blochmap_{\tilde\recip_q}\bigl( \Gen^\ms(\tilde \hu)\bigr) = \Gen^\rp(t_q\tilde\hu).
\end{align}
Suppose $O^\cf \in \OpConfig^\cf$ is constructed from the set $(g, \Gen^\cf(\hu_1),\cdots \Gen^\cf(\hu_n))$, $O_b^\rl \in \OpConfig^\rl$ is constructed from the set $(g, \Gen^\rl(t_b\hu_1),\cdots \Gen^\rl(t_b\hu_n))$, $O^\ms \in \OpConfig^\ms$ is constructed from the set $(g, \Gen^\ms(\tilde \hu_1),\cdots \Gen^\ms(\tilde \hu_n))$, and $O_q^\rp \in \OpConfig^\rp$ is constructed from the set $(g, \Gen^\rp(t_q\tilde \hu_1),\cdots \Gen^\rp(t_q\tilde \hu_n))$. Then
\begin{align}
& \blochmap_{\recip_b}( O^\cf) = O_b^\rl, \\
& \blochmap_{\dm{\tilde\recip_q}}( O^\ms) = O_q^\rp.
\end{align}
\end{thm}
\begin{proof}
The proof is in Section \ref{proof:ergodicunfolding}.
\end{proof}
The theorem above relates reciprocal space with momentum space, and real space with configuration space through a similarity transform, where the transform describes ergodic sampling. This completes the two horizontal legs of the \dm{isomorphic} diagram as shown in Figure \ref{figure:commutative}.
To complete the two vertical legs of the diagram, we must define the relationship between real and momentum spaces, or reciprocal and configuration spaces.
These two relationships are understood through the Bloch transforms:
\begin{thm}
\label{thm:blochtransform}
For $\hu \in \UnderlyingSpace^\cf_\herm$, we have
\begin{align}
& \blochmap_\G \bigl( \Gen^\rl(\hu) \bigr) = \Gen^\ms(\tilde \hu), \\
& \blochmap_{\tilde\G} \bigl( \Gen^\rp(\tilde \hu)\bigr) = \Gen^\cf(\hu).
\end{align}
Suppose $O^\cf \in \OpConfig^\cf$ is constructed from the set $(g, \Gen^\cf(\hu_1),\cdots \Gen^\cf(\hu_n))$, $O^\rl \in \OpConfig^\rl$ is constructed from the set $(g, \Gen^\rl(\hu_1),\cdots \Gen^\rl(\hu_n)),$ $O^\ms \in \OpConfig^\ms$ is constructed from the set $(g, \Gen^\ms(\tilde \hu_1),\cdots \Gen^\ms(\tilde \hu_n))$, and $O^\rp \in \OpConfig^\rp$ is constructed from the set $(g, \Gen^\rp(\tilde \hu_1),\cdots \Gen^\rp(\tilde \hu_n))$. Then
\begin{align}
& \blochmap_\G( O^\rl)= O^\ms, \\
& \blochmap_{\tilde \G}( O^\rp) = O^\cf.
\end{align}
\end{thm}
\begin{proof}
The proof is in Section \ref{proof:blochtransform}.
\end{proof}
With the \dm{isomorphic} map defined, we next state local site and local momenta sampling formulations for observables in configuration and momentum space:
\begin{thm}
\label{thm:representation}
Suppose $O^\cf \in \OpConfig^\cf$ is constructed from the set $(g, \Gen^\cf(\hu_1),\cdots \Gen^\cf(\hu_n))$, $O_b^\rl \in \OpConfig^\rl$ is constructed from the set $(g, \Gen^\rl(t_b\hu_1),\cdots \Gen^\rl(t_b\hu_n))$, $O^\ms \in \OpConfig^\ms$ is constructed from the set $(g, \Gen^\ms(\tilde \hu_1),\cdots \Gen^\ms(\tilde \hu_n))$, and $O_q^\rp \in \OpConfig^\rp$ is constructed from the set $(g, \Gen^\rp(t_q\tilde \hu_1),\cdots \Gen^\rp(t_q\tilde \hu_n))$. Then
\begin{align}
& \Tr \; O^\cf = \nu \biggl(\sum_{\alpha \in \A_1} \int_{\Gamma_2} [O_b^\rl]_{0\alpha,0\alpha}db + \sum_{\alpha \in \A_2} \int_{\Gamma_1} [O_b^\rl]_{0\alpha,0\alpha}db\biggr), \\
& \Tr \; O^\ms = \nu^* \biggl( \sum_{\alpha \in \A_1} \int_{\Gamma_1^*} [O_q^\rp]_{0\alpha,0\alpha}dq + \int_{\Gamma_2^*} [O_q^\rp]_{0\alpha,0\alpha}dq\biggr).
\label{e:ms_integral}
\end{align}
\end{thm}
\begin{proof}
The proof is in Section \ref{proof:representation}.
\end{proof}
Finally, we prove the equivalence of observables in all spaces:
\begin{thm}
\label{thm:observables}
For $O^\name \in \OpConfig^\name$ with hopping functions $\hu \in \UnderlyingSpace^\cf_\herm$ for real and configuration spaces, and $\tilde \hu$ for reciprocal and momentum spaces, we have
\begin{equation}
\TrLim \; O^\rp = \Tr \; O^\ms = \TrLim \; O^\rl = \Tr \; O^\cf.
\end{equation}
\end{thm}
\begin{proof}
The proof is in Section \ref{proof:observables}.
\end{proof}
}

\subsection{Hamiltonian with Mechanical Relaxation Effects}
\label{sec:relax}

Mechanical relaxation occurs when the atoms relax from their perfect homogenous configurations to minimize their energy due to the presence of the neighboring layer (see Figure \ref{fig:relaxation}). These effects are of relevance in the regime where there is a large moir\'{e} pattern, which occurs when the two lattices are nearly aligned.
\begin{figure}[ht]
\centering
\includegraphics[width=.6\textwidth]{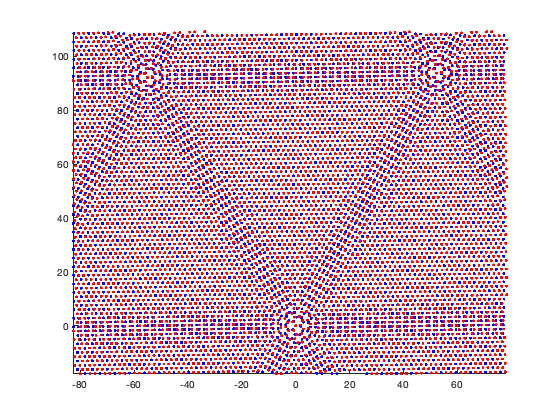}
\caption{Atomic positions for a twisted bilayer of MoS$_2$ ($\theta = 3^\circ$) after mechanical relaxation. The relaxation was computed according to the model in Ref.~\cite{relaxphysics18}. To generate this image, the interlayer coupling strength is exaggerated by a factor of $100$ so that the displacement effects are easily visible even at this relatively large twist angle.}
\label{fig:relaxation}
\end{figure}
\begin{definition}
We define the incommensurate Brillouin Zone 
\[\Gamma_{21}^* := \dm{2\pi(A_2^{-T}-A_1^{-T})[0,1)^2}\]
with lattice matrix
\[
\Theta_{21} = 2\pi (A_2^{-T}-A_1^{-T}).
\]
\end{definition}
\begin{assumption}
\label{assump:small}
We assume the inverse moir\'{e} scale is small, which we denote
\begin{equation*}
\theta :=  \|\Theta_{21}\|_\op \ll 1.
\end{equation*}
\end{assumption}
Each lattice site $R_j \in \R_j$ will be displaced. The assumption made in \cite{relaxphysics18,cazeaux2018energy} is that this displacement will be regular in configuration, which is a reasonable approximation given that the local geometry varies smoothly in shift. 
%However, as in the Frenkel-Kontorova model, we can't expect this to be exact, as there the displacement field %can become a cantor-like function with respect to configuration [citations]. However, to understand the %materials and compute electronics, this modeling approximation is reasonable.
We  consider displacement fields of each sheet 
\begin{align}
    &u_1 : \Gamma_{2} \rightarrow \mathbb{R}^2,\\
    &u_2 : \Gamma_{1} \rightarrow \mathbb{R}^2.
\end{align}
 The orbital dependence is important in the modeling, as the orbitals have different spatial locations, and thus configurations. \dm{We also will use the periodic extensions of $u_1$ and $u_2$ so they can be defined over the domain $\mathbb{R}^2$.}  To describe how $R_j \in \R_j$  changes position under mechanical relaxation, one simply uses the mapping
\begin{equation*}
R_j \mapsto R_j + u_{j}(R_j).
\end{equation*}
To find derivations and modeling of the $u_j$'s, see \cite{relaxphysics18,cazeaux2018energy}. We also outline the modeling in Appendix \ref{app:relax}. \dm{We reiterate here that this displacement field is taken as an input to our model, and is not calculated in this work.}

\dm{We let $h^{ij}$ be the tight-binding coupling function between sites on sheet $i$ and $j$ dependent on the vector distance between sites. The real space Hamiltonian has matrix elements:
\begin{align*}
&H_{R\alpha,R'\alpha'} = [h^{11}]_{\alpha\alpha'}(R + u_1(R) - R' - u_1(R')), & R\alpha,R'\alpha' \in \Omega_1, \\
&H_{R\alpha,R'\alpha'} = [h^{22}]_{\alpha\alpha'}(R + u_2(R) - R' - u_2(R')), & R\alpha,R'\alpha' \in \Omega_2, \\
&H_{R\alpha,R'\alpha'} = [h^{21}]_{\alpha\alpha'}(R + u_2(R) - R' - u_1(R')), & R\alpha \in \Omega_2,R'\alpha' \in \Omega_1,\\
&H_{R\alpha,R'\alpha'} = [h^{12}]_{\alpha\alpha'}(R + u_1(R) - R' - u_2(R')), & R\alpha \in \Omega_1,R'\alpha' \in \Omega_2.
\end{align*}
%\cml{Make the above and below consistent with reciprocal lattice version}
Our next task is to interpret this Hamiltonian in configuration space so we can write the hopping functions $\hu \in \hopspace_\herm^\cf$.
 We start by simple algebraic manipulations of the intralayer hoppings
\begin{equation*}
[h^{11}]_{\alpha\alpha'}(R + u_1(R) - R' - u_1(R')) = [h^{11}]_{\alpha\alpha'}(R-R' + u_1(R )- u_1(R - (R - R'))
\end{equation*}
and interlayer hoppings
\begin{equation*}
[h^{21}]_{\alpha\alpha'}(R + u_2(R) - R' - u_1(R'))=[h^{21}]_{\alpha\alpha'}(R-R' + u_2(R-R') - u_1(-(R-R'))).
\end{equation*}
Parallel equations hold for the other sheet couplings, and so we readily find
\begin{align*}
&h_R^1(b) = h^{11}( R + u_1(b) - u_1(b-R)), & R \in \R_1, \\
&h_R^2(b) = h^{22}( R + u_2(b) - u_2(b-R)), & R \in \R_2.
\end{align*}
We then get the underlying functions $\hu$ by:
\begin{align*}
&\hu_{11}(q,b) = \sum_{R \in \R_1} e^{iq \cdot R} h_R^1(b), \\
&\hu_{22}(q,b) = \sum_{R \in \R_2} e^{iq \cdot R} h_R^2(b), \\
&\hu_{21}(x) = h^{21}(x + u_2(x) - u_1(-x)), \\
&\hu_{12}(x) = h^{12}(x + u_1(x) - u_2(-x)).
\end{align*}
Hence we have the mechanically relaxed tight-binding model
\[
H = \Gen^\rl(\hu).
\]
To consider a momentum or reciprocal space formulation, we now only need to construct $\tilde \hu$ and use the above procedure to understand how to populate the momenta hopping terms.
}

\begin{figure}[ht]
\includegraphics[width=.8\textwidth]{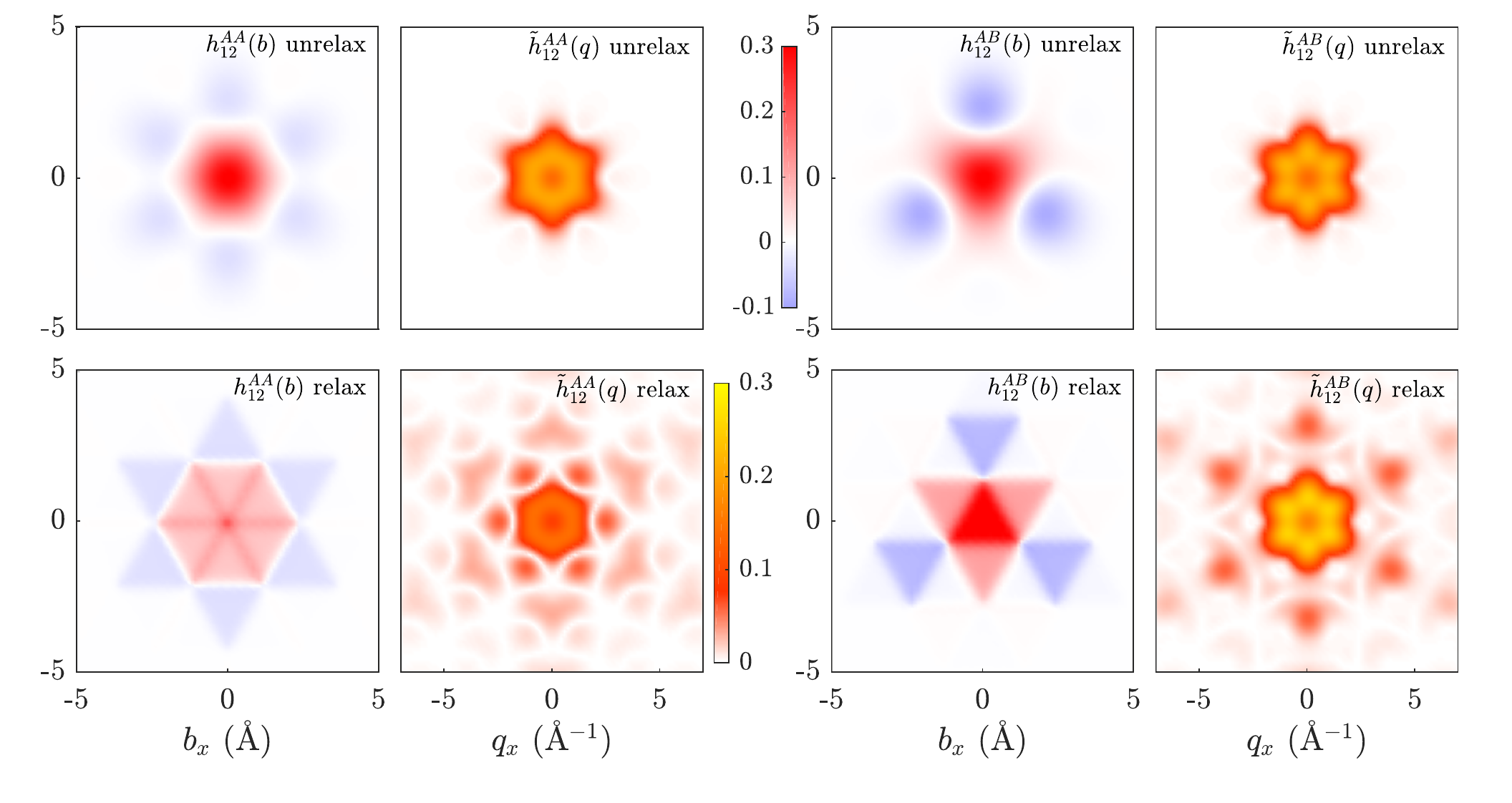}
\caption{ Here we show the interlayer coupling \dm{of TBG} for a small twist angle $\theta = 0.3^\circ$ in both real and momentum space. $A$ and $B$ here refer to the two orbitals associated with a sheet of graphene, i.e. orbital index sets $\A_1$ and $\A_2$ each consist of orbitals $A$ and $B$. The first column is AA coupling in real space, while the second column is the magnitude of AA coupling in momentum space. Similarly for columns three and four, but for AB coupling. Across all four columns, the top row is without relaxation while the bottom row is with relaxation.}
\label{fig:coupling}
\end{figure}

\section{Numerical Method}
\label{sec:algorithm}

The principle gain of using momentum space that we outline here is that all interlayer and mechanical relaxation effects are arising from weak Van der Waals forces, which allows us to consider them as a form of perturbation to the two monolayer structures. This perturbative idea must be treated with great care, however, as the perturbative effect arises from the relationship between energies and momenta in the monolayer band structures, the short hopping distance in momenta due to the inverse moir\'{e} scale being small, and the hopping strength and range along the reciprocal lattices. We outline this section as follows: first we will discuss the energy and momenta relationship arising from the monolayer band structure. Then we will construct the appropriate truncation of the discrete operators $\Gen^\rp(t_q\tilde \hu_j)$ to generate a set of finite matrices $H_j(q)$. If we have an observable $O^\ms$  represented by $(g, \Gen^\ms(\tilde \hu_1),\cdots \Gen^\ms(\tilde \hu_n))$, then we will be able to write our approximation of $\Tr\, O^\ms$ motivated by \eqref{e:ms_integral} as an integral over the incommensurate Brillouin Zone instead of the two reciprocal lattice unit cells $\Gamma_1^*$ and $\Gamma_2^*$, and with finite matrices $H_j(q)$ replacing $\Gen^\rp(t_q\tilde \hu_n)$. We will then conclude with the convergence result of observables.

A key point here is that the spectral information coming from a Hamiltonian $H$ is now captured by a corresponding family of finite matrices $H(q)$ for $q \in \dm{\Gamma^*_{21}}$, which instantly gives a band structure representation. This provides the traditional physical insight into the relationship between energy and \dm{momenta} in a crystalline material. Most observables can be computed from the spectral information of a single Hamiltonian, so for the rest of the section we will assume we are considering a single Hamiltonian $H = \Gen^\rl(\hu)$. We conclude the section with a numerical convergence result for the energy eigenstates in the center of the energy window of interest, and comment on how to approach general observables. We do not derive a general convergence result for observables, as the convergence is modified by both the observable's dependence on spectral properties and the energy landscape of the monolayer band structures.

\subsection{Energy and Momentum Selection}

We already assumed a long moir\'{e} length scale. Our next assumption is that we are only interested in a specific range of energies, which typically will be a small subset of the monolayer Hamiltonians' spectra. We call $H_1$ and $H_2$ the unrelaxed periodic monolayer Hamiltonians over $\X_1^\rl$ and $\X_2^\rl$ respectively. We denote their momentum space counterparts $\mon_1(q)$ and $\mon_2(q)$ as multiplication operators over $\X_1^\ms$ and $\X_2^\ms,$ respectively. Let $\sigma_j(q)$ be the set of eigenvalues of $\mon_j(q)$. We denote 
\[ \mon = \begin{pmatrix} \mon_1 & 0 \\ 0 & \mon_2\end{pmatrix}\]
as the two monolayer operators defined over $\X^\ms$. Let
\begin{equation}
    \energy \subset \mathbb{R}
\end{equation}
be an open set that corresponds to the energy region of interest. To focus on this energy region with relaxation (or phonon) and interlayer coupling effects, however, it is more convenient to expand the energy region we consider by a width equal to something a little larger than \dm{twice} the strength of the coupling, which we denote
\begin{equation}
\label{e:eta}
\eta := (2+\alpha)\| \Gen^\rp(\tilde \hu) - \Gen^\rp(\mon)\|_\op.
\end{equation}
Here $\alpha > 0$ so that $\eta$ is larger than twice the coupling strength. 
We denote the ball in $\mathbb{R}$ of radius $\eta$ as $B_\eta$. Then our extended energy region is $\energy + B_\eta$.
We find there are corresponding momenta regions in $\Gamma_1^*$ and $\Gamma_2^*$ that correspond to this energy region:
\begin{equation}
\Gamma_j^*(\energy+B_\eta) = \{ q \in \Gamma_j^* : \sigma_j(q) \cap (\energy{+} B_\eta) \neq \emptyset\}.
\end{equation}
%We will occasionally use $\Gamma_j^*(A)$ for other sets $A \subset \mathbb{R}$.
As an example of the energy and momenta relationship, see Figure \ref{fig:energy_momenta}.
\begin{figure}[ht]
\includegraphics[width=.8\textwidth]{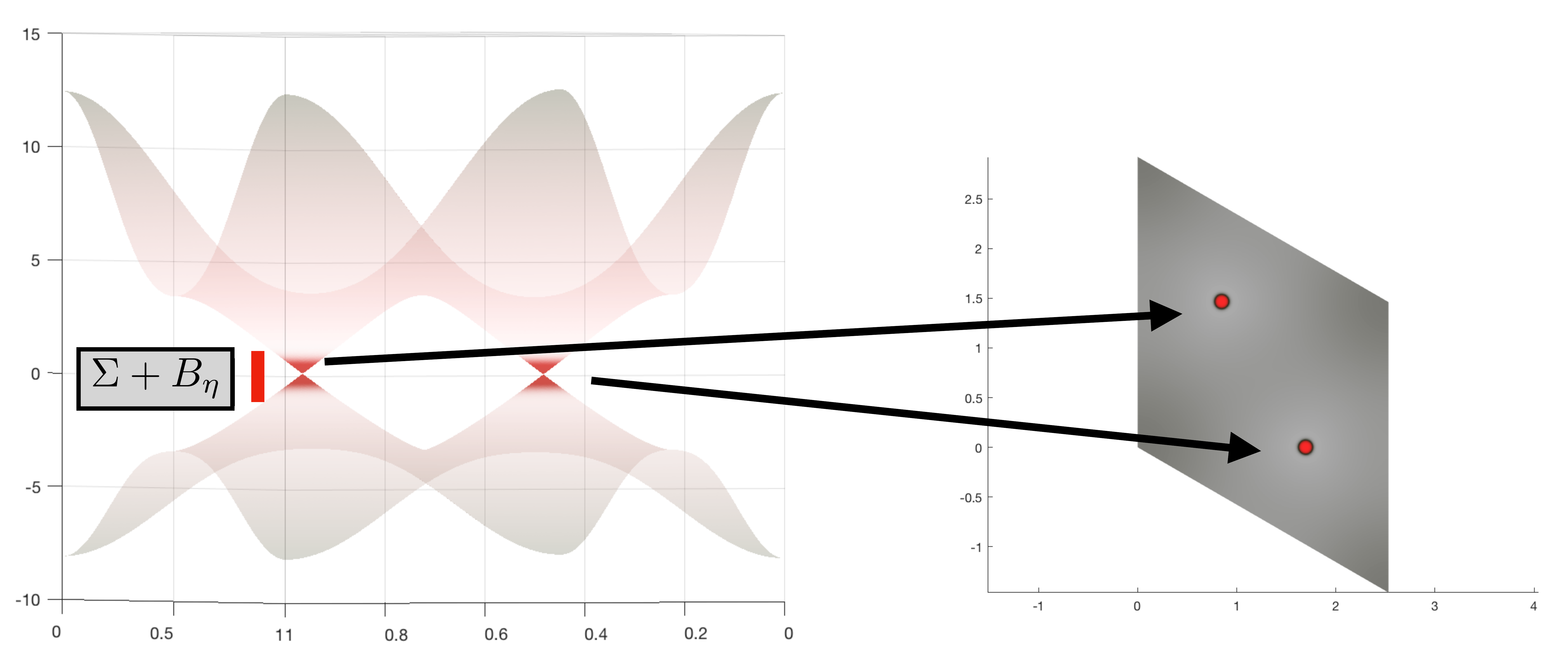}
\caption{ For a two-band graphene model, we show the relationship between an energy region $\energy$ and momenta for the energy region near the Fermi energy. On the left is a graphene monolayer band structure, and on the right is the reciprocal lattice unit cell of monolayer graphene. The highlighted region on the right corresponds to $\Gamma_1^*(\energy + B_\eta).$}
\label{fig:energy_momenta}
\end{figure}

\subsection{Reciprocal Space Truncation} 

For the numeric study of observables and related objects, it is also useful to define spaces of operators with a specified exponential decay rate in their hopping functions.
We define a secondary space $\OpConfig^\cf_\herm(\gamma,\tilde \gamma) \subset \OpConfig_\herm^\cf$  to be operators with hopping functions satisfying:
\begin{align*}
&\bigl|[\hu_{jj}]_{RG}\bigr| \leq c e^{-\gamma |R| -\tilde \gamma |G|}, 
&|\hu_{ij}(x)| \leq ce^{-\gamma |x|}, \hspace{1.6cm} |\hat\hu_{ij}(\xi)| \leq  ce^{-\tilde \gamma |\xi|},
\end{align*} 
for some $c>0$. The $\gamma$ controls the decay in orbital hopping distance, while $\tilde \gamma$ controls the regularity of the hoppings as a function of configuration.
We likewise define $\OpConfig^\ms_\herm(\tilde \gamma, \gamma) \subset \OpConfig_\herm^\ms$ with hopping functions $\tilde \hu$ by
\begin{align*}
&\bigl|[\tilde \hu_{jj}]_{GR}\bigr| \leq c e^{-\tilde \gamma |G| - \gamma |R|}, 
&|\tilde \hu_{ij}(\xi)| \leq ce^{-\tilde  \gamma |\xi|}, \hspace{1.6cm} |\check{\tilde \hu}_{ij}(x)| \leq  ce^{- \gamma |x|}.
\end{align*}
The $\gamma$ controls regularity of the hoppings in terms of momenta, while the $\tilde \gamma$ controls hopping distance along the reciprocal lattices.
Here $\check{\tilde \hu}_{ij}$ refers to inverse Fourier transform of $\tilde \hu_{ij}$. To better understand how the twist angle affects the regularity, we define 
\begin{equation*}
    \begin{split}
\OpConfig_\varepsilon^\name&(\gamma_1,\gamma_2) = \\
&\biggl\{ \int_\contourprod g(z)\prod_{j=1}^n(z_j-H_j)^{-1}dz : n > 0, \; H_1,\cdots H_n \in \OpConfig^\name_\herm(\gamma_1,\gamma_2), \; C_j \in \Contour_\varepsilon(H_j),\; g \in \Lambda(\contourprod)\biggr\}.
    \end{split}
\end{equation*} 
It is observed in \cite{relaxphysics18,cazeaux2018energy} that relaxation in configuration space becomes sharper proportional to $\theta^{-1}$. \dm{We assume} there exists $\gamma,\tilde \gamma > 0,$ independent of $\theta,$ such that
\begin{align}
\label{e:relax_hop1}
 & \Gen^\cf(\hu) \in \OpConfig^\cf(\gamma , \tilde \gamma\theta),\\
 & \Gen^\ms(\tilde \hu) \in \OpConfig^\ms(\tilde\gamma\theta,\gamma).
 \label{e:relax_hop2}
\end{align}
\dm{This assumption is motivated by \cite{relaxphysics18}.}
In other words, configuration space methods suffer a loss of regularity with respect to configuration, while momentum space suffers with slower reciprocal space localization (see Figure \ref{fig:coupling} for interlayer hopping functions in momentum and configuration spaces with and without mechanical relaxation effects).

When considering $\Gen^\rp(t_q\tilde \hu)$, we note each basis element of $\X^\rp$, denoted $G\alpha$, is associated with a wavenumber $q+G$. It likewise has a corresponding energy set $\sigma_j(q+G)$ where $j$ is the sheet index corresponding to site $G\alpha$. We define for an energy region $A \subset \mathbb{R}$ and $r > 0$ a corresponding set of basis elements:
\[
\Omega^*_r(q, A) = \{ G\alpha \in \Omega_j^* : q + G \in \Gamma_j^*(A)+B_r, \; j = 1,2 \}.
\]
Here we consider $\Gamma_j^*(A)+B_r$ as a subset of the torus $\Gamma_j^*$, and so $q+G$ is modulated by the torus $\Gamma_j^*$ in the definition above. For $U \subset V \subset \Omega^*$, we define
\[
J_{V \leftarrow U} : \ell^2(U) \rightarrow \ell^2(V)
\]
Choice of $r$ controls our accuracy.
As seen in Figure \ref{fig:dof_space}, for appropriate systems, $\Omega_r^*(q, \energy + B_\eta)$ consists of isolated regions. Regions can be defined as connected or isolated over reciprocal space, but here we skip the technical definition as we consider it reasonably intuitive, and simply cite \cite{momentumspace17} for the details of connectedness.
We define $\mathcal{B}(\Omega^*)$ as the collection of subsets of $\Omega^*$. Then we define $I : \mathcal{B}(\Omega^*) \rightarrow \mathcal{B}(\Omega^*)$ as the operation that maps a subset of $\Omega^*$ to its isolated degrees of freedom  containing a site $0\alpha$ for any $\alpha$. As an example, see the circled red region in Figure \ref{fig:dof_space} corresponding to $I\bigl( \Omega_r^*(q,\energy+B_\eta)\bigr)$.

We will let $\tau$ be a truncation in hopping distance on $\Gen^\rp(t_q\tilde \hu)$ by using new hopping functions
\begin{align*}
& \tilde \hu^{(\tau)}_{11}(b,q)= \sum_{G \in \R_2^* \cap B_{\tau}} h_{G}(q)e^{ i b \cdot G}, \\
& \tilde \hu^{(\tau)}_{22}(b,q)= \sum_{G \in \R_1^* \cap B_{\tau}} h_{G}(q) e^{ i b \cdot G}, \\
& \tilde \hu^{(\tau)}_{ij}(\xi) = \chi_\tau(\xi) \tilde \hu_{ij}(\xi).
\end{align*}
Here $\chi_\tau(x)$ is smooth and $1$ on $[-\tau+1/2,\tau-1/2]$, and compactly supported on $[-\tau,\tau]$.
With $\tau$ selected, we have the new Hamiltonian $\Gen^\rp(t_q\tilde \hu^{(\tau)})$. We note this operator does not actually exist in $\OpConfig^\rp$ as the hopping functions are not in $\hopspace^\ms_\herm$. To generate a finite matrix corresponding to the energy window of interest, we consider
\begin{align}
&\dof_{r}(q) = I\bigl(\Omega_r^*(q,\energy+B_\eta)\bigr),\\
& J_{r}(q) = J_{ \Omega^* \leftarrow \dof_r(q) },\\
& H_r(q) = J_{r}^*(q) \Gen^\rp(t_q\tilde \hu^{(\tau)}) J_{r}(q).
\end{align}
This is a finite matrix over the basis space ${I\bigl(\Omega_r^*(q,\energy+B_\eta) \bigr)}$ as long as the region selected by $I$ was an isolated collection of a {finite} number of elements. This is where the relationship between the energy window of interest and the monolayer band structure becomes vital. If {$\Gamma_j^*(\energy+B_\eta)+B_r \subset \Gamma_j^*$} is not homotopically trivial, then {$\Omega_r^*(q,\energy+B_\eta)$} is not finite for $q \in \Gamma_j^*(\energy+B_\eta)+B_r$. We note that if $q \not\in {\Gamma_j^*(\energy+B_\eta)+B_r}$, then {$\Omega_r^*(q)$} is empty.

\begin{figure}[ht]
\begin{subfigure}{.5\textwidth}
\includegraphics[width=.9\textwidth]{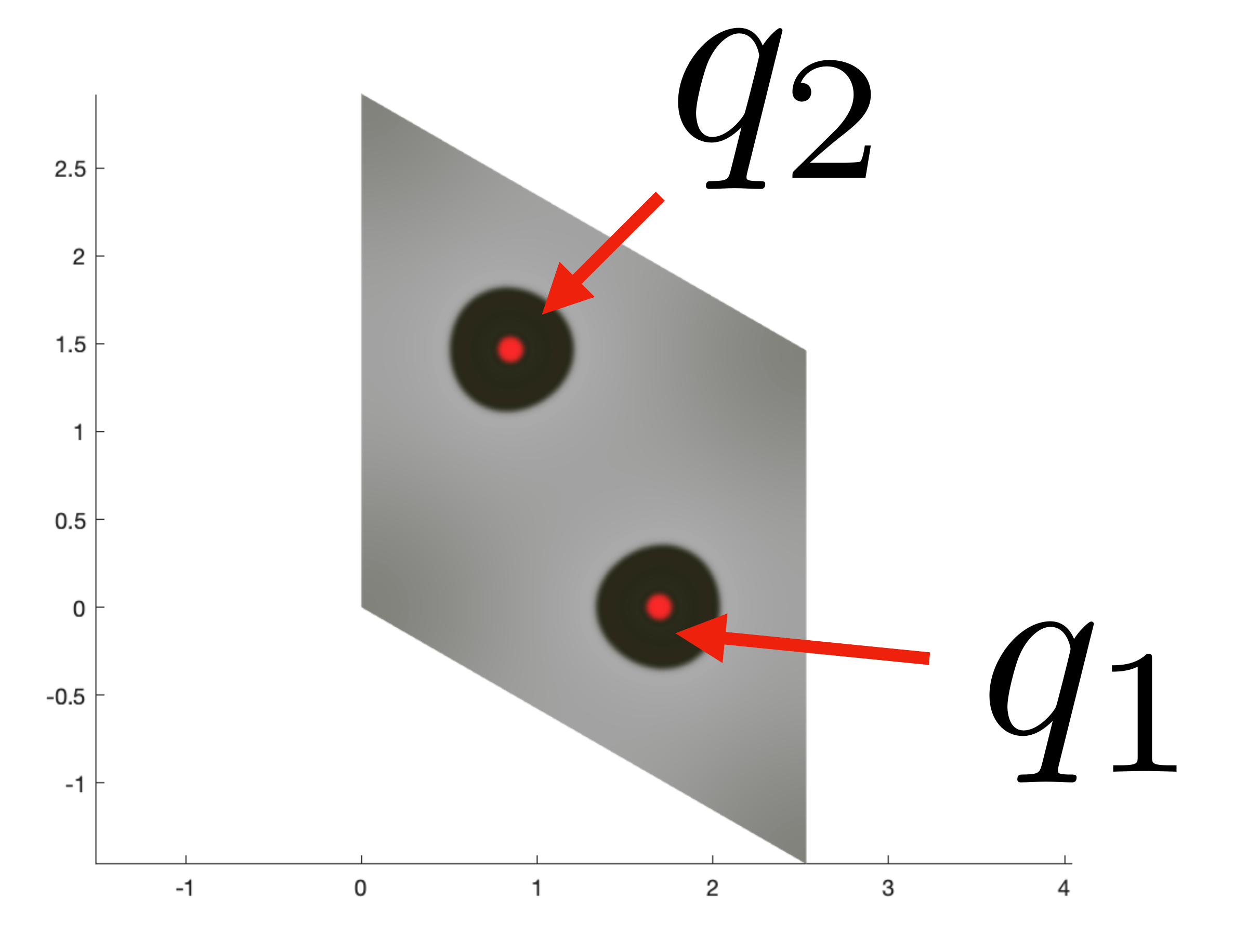}
\caption{\dm{The central regions are {$\Gamma_1^*(\energy + B_\eta)$}, and adding the dark regions around them we obtain {$\Gamma_1^*(\energy + B_\eta) + B_r$}. Here $q_1$ and $q_2$ in the central regions are considered starting points. }}
\vspace{.9cm}
\label{fig:extended_bz}
\end{subfigure}
\hspace{5mm}
\begin{subfigure}{.45\textwidth}
\includegraphics[width=.9\textwidth]{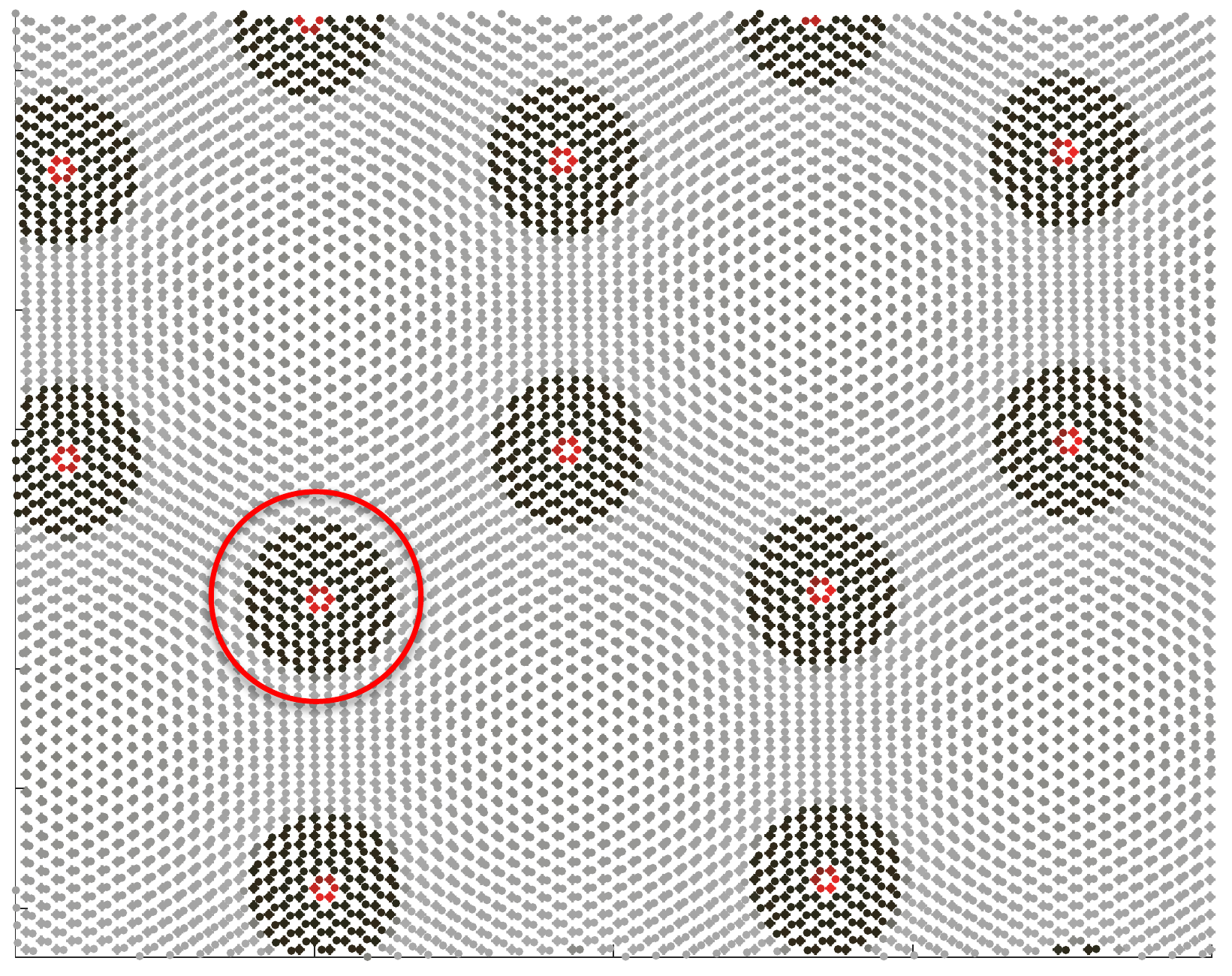}
\caption{Here we plot reciprocal space with each reciprocal lattice site color-coded according to the corresponding wavenumber region. The central circled region is the region corresponding to a starting point \dm{$q_1$ containing lattice site $0$}.}
\label{fig:reciprocal_space}
\end{subfigure}
\caption{Momentum space and reciprocal space correspondence for twisted bilayer graphene. }
\label{fig:dof_space}
\end{figure}

We observe the following symmetry, which is useful for recognizing the set of starting momenta of interest is the incommensurate Brillouin Zone:
\begin{prop}
\label{prop:inc_BZ}
If $G_j = (-1)^{j}2\pi A_j^{-T} n$ for $n \in \mathbb{Z}^2$, $G_j' \in \R_j^*$ and $\tilde G_\ell \in \R_\ell^*$, and 
\[
\delta q = \Theta_{21}n = 2\pi (A_2^{-T} - A_1^{-T})n,
\]
then
\[
[H_r(q)]_{(G_j+G_j')\alpha,(G_\ell + \tilde G_\ell)\alpha'} = [H_r(q+\delta q)]_{G_j'\alpha,\tilde G_\ell\alpha'},
\]
as long as $q, q+\delta q \in { \Gamma_j^*(\energy+B_\eta)+B_r}$ for $j = 1,2$.
\end{prop}
\begin{proof}
Proof is in Section \ref{proof:prop}.
\end{proof}
In other words, $H_r(q)$ is an identical matrix when shifted along the lattice  defined by the incommensurate Brillouin Zone, aside from a simple relabeling of the basis elements. As a consequence, we will select one momenta per isolated region \dm{labeled} $q_1,\cdots q_n$. In the case of bilayer gaphene, $n = 2$ and $q_1$ and $q_2$ correspond to the two unique Dirac points. The two corresponding Dirac cones are often \refchange{referred} to as the two ``valleys,'' and they are related by a time-reversal symmetry. We first consider convergence of the density of states, as it is the simplest observable. The density of states is approximated by the following for \dm{spectral} resolution $\varepsilon$:

\begin{align}
& D_\varepsilon(E) = \TrLim \; \phi_\varepsilon(E - \Gen^\rl(\tilde \hu)),\\
& \phi_\varepsilon(E) = \frac{1}{\sqrt{2\pi\varepsilon}}e^{-E^2/2\varepsilon^2}.
\end{align}
\begin{thm}
\label{thm:convergence}
Consider incommensurate bilayer system as described above with long moir\'{e} length scale. Consider $E \in \energy$, and $\varepsilon \ll 1$. Let $\tau > 0$ be a hopping truncation. Then there are constants $\gamma_h$, $\gamma_m$, and $\gamma_g$ corresponding to hopping truncation error, momenta truncation error, and Gaussian decay rates respectively such that
\begin{equation}
    \biggl| D_\varepsilon(E) - D_{\varepsilon,r}(E) \bigg| \lesssim \varepsilon^{-3/2}(e^{-  \gamma_h  \tau} + \varepsilon^{-2}e^{ - \gamma_m r}+ e^{- \gamma_g \varepsilon^{-2}}),
\end{equation}
where 
\[
D_{\varepsilon,r}(E) = \dm{\sum_{j=1}^n\nu^*\int_{\Gamma_{21}^*+q_j} \Tr \; \phi_\varepsilon(E-H_r(q))dq.}
\]
When mechanical relaxation effects are not included, i.e.,
\begin{equation}
\Gen^\ms(\tilde \hu) \in \OpConfig^\ms(\tilde \gamma,  \gamma),
\end{equation}
then we have $\gamma_h$ is independent of $\theta$, but $ \gamma_m = O( \theta^{-1})$. Meanwhile if mechanical relaxation effects are included, i.e., 
\begin{equation}
\Gen^\ms(\tilde \hu) \in \OpConfig^\ms(\tilde \gamma \theta,  \gamma),
\end{equation}
then we have $\gamma_h = O(\theta)$ and $\gamma_m = O(1)$.
\end{thm}
\begin{proof}
The proof is in Section \ref{proof:convergence}.
\end{proof}
We note that $\tau$ is measured in reciprocal lattice length-scale, while $r$ is a radius of a ball living within a single reciprocal lattice unit cell.
When mechanical relaxation effects are included, momenta convergence looks like $e^{-\gamma_m r}$. 
\begin{remark}
    \refchange{If we take $r$ too large, we will lose homotopic triviality of our truncated region, which would lead to an infinite sized matrix. Hence there is a limit to how large we can take $r$ before we lose the ability to write a meaningful numeric scheme.} In practice, $\gamma_m$ is quite large even with mechanical relaxation effects included due to the weak Van der Waals coupling between sheets in bilayer 2D heterostructures. So while we can't converge to arbitrary accuracy with this approach, we can converge to a sufficiently small error that the resulting information is physically meaningful. In other words, isolated momenta regions must be considered coupled if you are interested in an arbitrarily small energy resolution scale $\varepsilon$, but most physically relevant questions don't require this level of accuracy in the answer.
\end{remark}

\begin{remark}
At this point, $H_r(q)$ for twisted bilayer graphene is very close to the Bristritzer-MacDonald model. The latter model can be derived from $H_r(q)$ by the following approach: First mechanical relaxation effects are neglected. Secondly, monolayer coupling terms $\mon_j$ are replaced with a Dirac cone recentered around the momenta $K$ corresponding to the tip of the Dirac point:  \dm{$m_j(K + q) \approx \begin{pmatrix} 0 & q_x + i q_y \\ q_x - i q_y & 0\end{pmatrix}$ where $q = (q_x,q_y)$}. Finally, the interlayer coupling near the Dirac point is approximated by just the three smallest scattering directions, with equal hopping strength, and all other hoppings are neglected. This makes an elegant model for approximating unrelaxed twisted bilayer graphene. Our method allows extension of this idealized model to more complex momentum space hopping effects, for example those produced by mechanical relaxation effects at small angles. Note from Figure \ref{fig:coupling} that the momentum space hopping range increases with moir\'{e} length scale. The approach outlined in this paper applies to materials other than twisted bilayer graphene, and to more complicated moir\'e systems such as strained interfaces and multilayered heterostructures.
\end{remark}
\begin{remark}
We note more complex observables can be computed by this method as well, as long as the observable is primarily dependent on spectral information in an energy window yielding  a collection of finite matrices $H_r(q)$. For example, the calculation of electrical conductivity via linear response using momentum space is treated in \cite{massatt2020efficient}.
\end{remark}

\begin{figure}[ht]
\includegraphics[width=.8\textwidth]{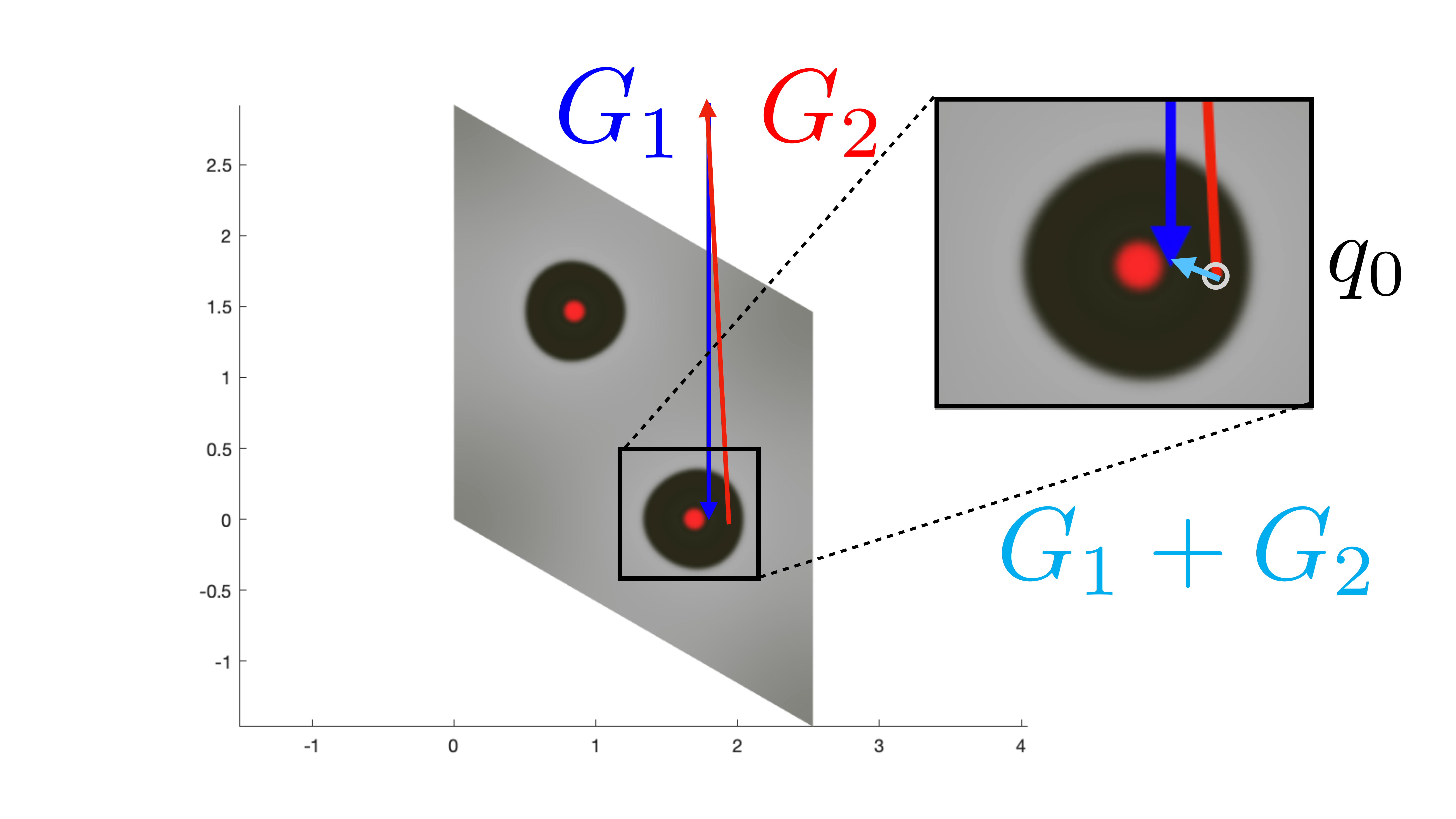}
\caption{ Here we consider a starting wavenumber $q_0 \in \Gamma_1^*$ and let $n \in \mathbb{Z}^2$. If $G_2 = 2\pi A_2^{-T} n$ with corresponding $G_1 = -2\pi A_1^{-T} n$, then the momenta $q+G_2$ can also be written as $q+ G_2+G_1 = q + \Theta_{21} n$. We thus see that $\Theta_{21}$ well characterizes the shift in momentum when we move along lattice sites in reciprocal space.}
\label{fig:moirecell}
\end{figure}

\section{Numerics}
\label{sec:numerics}

We now describe practical details on the implementation of our momentum basis model for relaxed TBG \cite{Catarina2019}, and present results on band structure and convergence.
At energies near the Fermi energy, monolayer graphene's band structure near the Brillouin zone corners $K$ and $K'$ can be described by the Dirac equation

\begin{equation}
H(q) = -\hbar v_F \vec{\sigma} \cdot \vec{q},
\end{equation}
with $\vec{q} = (q_x, q_y)$ and $\vec{\sigma} = (\sigma_1, \sigma_2)$, the first two Pauli matrices
\cite{Castro2009,Catarina2019}.
Importantly, the only free parameter is the Fermi velocity $v_F$, which sets the dispersion (slope) of the bands, and it is linearly dependent on the nearest-neighbor hopping parameter in the graphene tight-binding model.
However, as $q$ moves away from $K$, this simple model becomes less accurate due to missing terms from higher-order hopping parameters.
Instead, for our model we implement $H(q)$ as a Bloch wave defined by the tight-binding hopping parameters up to the fifth nearest neighbor, with values obtained from a previous first-principles study of graphene \cite{shiang2016}.

For TBG, the set of basis elements within the selected energy window, {$\Gamma_j^*(\Sigma+B_\eta)+B_r$}, can be described more succinctly in terms of a fixed truncation radius $\Lambda$.
This is because the conical bands of graphene lead to a linear relationship between spatial cutoffs and energy cutoffs.
As any basis element can be given in terms of a reciprocal lattice index $n \in \Z$, the truncation corresponds to keeping only those with position $|\Theta_{21} n| < \Lambda$.

For interlayer coupling, we use a direct plane-wave inner product to populate a specialized grid of momenta, and then perform quadratic interpolation to obtain the tunneling value at any generic $\xi$.
To ensure proper symmetry in the final Hamiltonian, one must ensure the sampling of both the real space and momentum space grids is consistent with the symmetries of the twisted bilayer.
For real space, the tunneling is sampled on a triangular lattice of points, with a smoothed radial cut-off at $8$ \AA{}.

We use an interlayer hopping functional derived from previous first-principle calculations~\cite{shiang2016,carr2018}.
We sample the plane-wave inner products ($\hat{h}$) on {a} truncated doubly-nested grid with moment $G + G'$, with $G \in \R^*_2$ and $G' \in \R^*_1$.
We truncate this model by only considering $|G|,|G'| < \tau$, for some truncation radius $\tau$, which limits both the large and small grid samplings.
For small twist angles, this truncation procedure leads to a momentum sampling which consists of ``islands'' sampled at the moir\'e reciprocal lattice scale, separated from one other on the monolayer reciprocal lattice scale.

The large (monolayer) scale sampling captures the rough scattering direction, with the three smallest such scatterings corresponding to the three tunneling directions in the BM model. 
The small (moir\'e) scale sampling captures the gradient of $\hat{h}$ in the vicinity of each scattering direction, a correction to the BM model which causes most of the particle-hole asymmetry in the low-energy bands of TBG~\cite{Carr2019exact}.
For any given $q + G + G'$, we then interpolate the tunneling strength from this customized grid of pre-calculated values of $\hat{h}_{\alpha \alpha'}$.
We note that if instead one attempts a square-grid 2D FFT, a large amount of symmetry and resolution inaccuracies occur even with very fine mesh-sizes, requiring large amounts of memory for poor performance.

One final implementation detail is related to the sublattice orbital shifts between twisted layers.
For each pair of interlayer orbitals, we redefine the sampling grid such that the two orbitals are aligned at the origin, and then add the relative shift $\Delta$ as an additional phase factor after interpolation by multiplying by $e^{i (q + G + G') \cdot \Delta}$.
For example, tunneling between two $A$ orbitals would require $\Delta = 0$, but between an $A$ and $B$ orbital $\Delta$ would be roughly the sublattice bonding distance (with an appropriate small twist for the layer of the $B$ orbital).

\begin{figure}[ht]
\includegraphics[width=.5\textwidth]{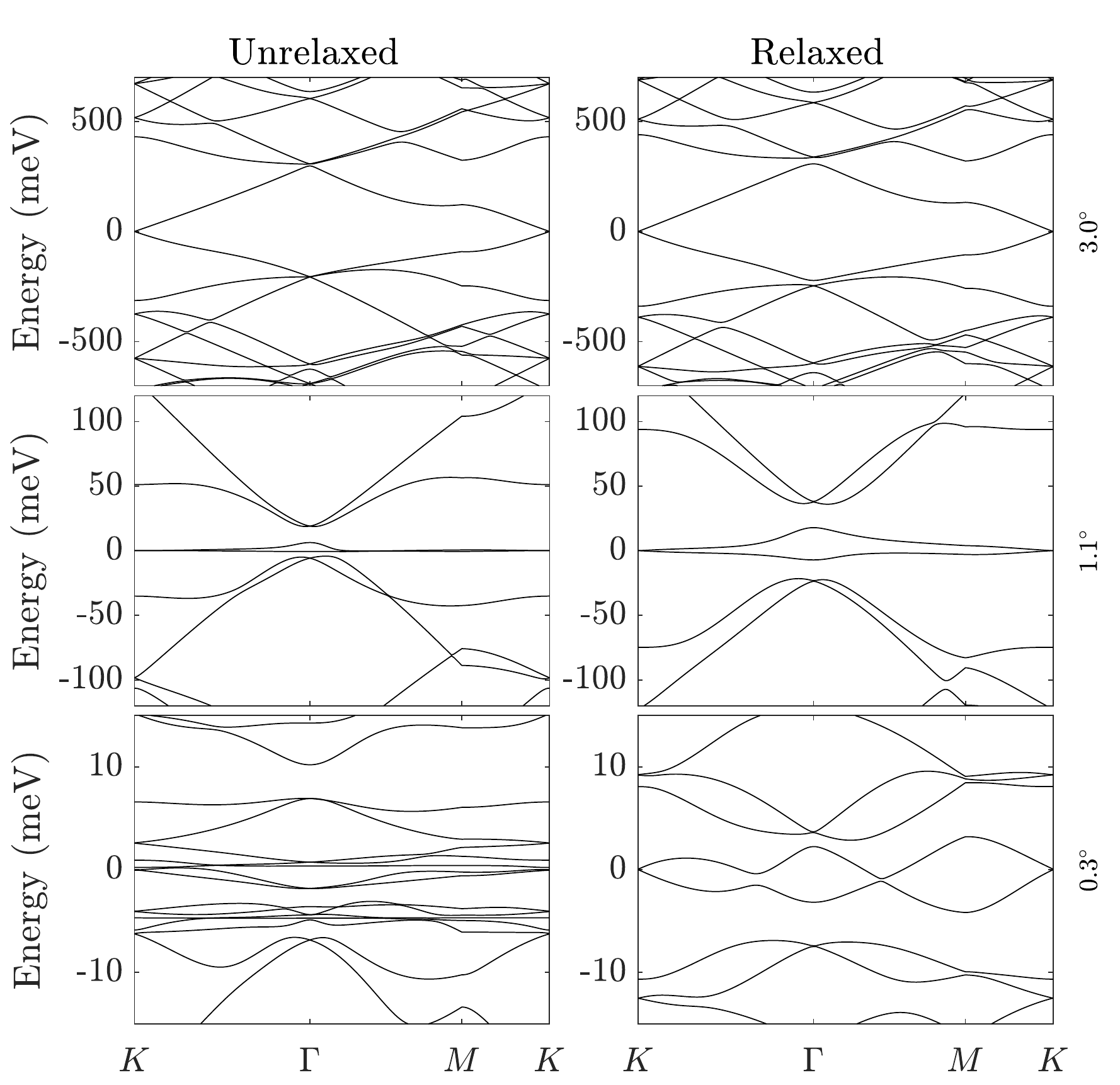}
\caption{Electronic band structure along high-symmetry lines of the moir\'e Brillouin zone at a single monolayer $K$ valley for three twist angles, $3.0^\circ$ (top), $1.1^\circ$ (middle), and $0.3^\circ$ (bottom). The first column shows the band structure for unrelaxed TBG, while the second shows that of relaxed TBG. The momentum axes are labeled in terms of the high-symmetry points of the reciprocal lattice of the moir\'e supercell, not the graphene monolayer cells.}
\label{fig:band_compare}
\end{figure}

With our truncation of the momentum basis defined and all relevant intra and inter-layer terms calculated, we can now diagonalize the Hamiltonian matrix to obtain electronic band structure.
In Fig. \ref{fig:band_compare}, we show results of our model for a single valley of TBG \cite{Catarina2019} for three angles, both relaxed and unrelaxed.
We see that at large angles ($\theta = 3.0^\circ$), the Dirac cones of graphene are still clearly visible.
The effects of relaxation are small but noticeable: a small moir\'e band gap opens up near the first band crossing at $\pm 350$ eV.

Near the magic angle ($\theta = 1.1^\circ$) \cite{Castro2009,Catarina2019}, the linear dispersion of the Dirac cone is nearly perfectly compensated by the interlayer band hybridization, creating an extremely flat band.
After relaxations, the band is slightly less flat and the moir\'e band gaps near $\pm 40$ eV are larger.
The flat bands are still possible for the relaxed system, but they are now at a slightly larger angle, due to an increase in the effective AA interlayer tunnel strength (see Fig. \ref{fig:coupling}).

At small angles ($\theta = 0.3^\circ$), accurately capturing atomic relaxation becomes of upmost importance.
As the moir\'e pattern is now many tens of nm, large domains of uniform AB or BA stacking occur and are criss-crossed by narrow domain-walls of intermediate stacking.
The unrelaxed band structure does not capture this atomic reconstruction, and shows a large amount of intersecting bands at low-energy.
With relaxations, the electronic structure is less busy at low energy, and clear Dirac points are still visible along with small moir\'e band gaps at roughly $\pm 5$ meV.

\begin{figure}[ht]
\includegraphics[width=.9\textwidth]{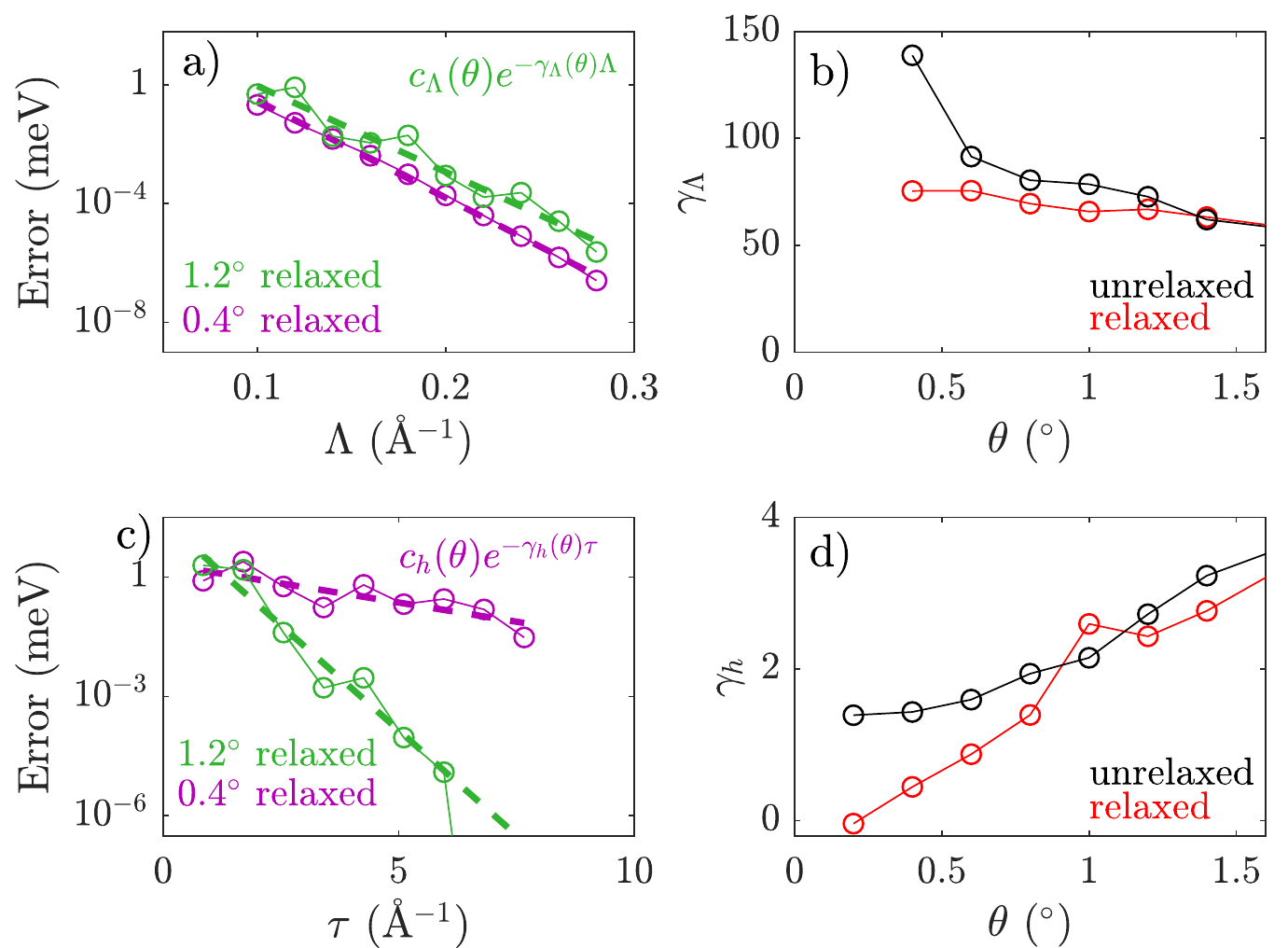}
\caption{Convergence of electronic structure.
\textbf{a)} The relative error for the $\Gamma$-point electron eigenvalue closest to the Fermi energy as a function of the momentum basis truncation radius $\Lambda$. The results for the relaxed system at $0.4^\circ$ and $1.0^\circ$ are in purple and green respectively. A linear fit to the log of the error is shown, giving a constant $c_\Lambda$ and slope $\gamma_\Lambda$.
\textbf{b)} The dependence of $\gamma_\Lambda$ on the twist angle $\theta$ for the unrelaxed (black) and relaxed (red) calculations. 
\textbf{c,d)} Same as (a,b) but for the convergence in the interlayer truncation radius $\tau$ and its exponential convergence $\gamma_h$.}
\label{fig:convergence}
\end{figure}

The exponential convergence of our relaxed momentum-space algorithm can be directly assessed by calculating the relative convergence of the eigenvalue at $q = 0$ of the moir\'e Brillouin Zone ($\epsilon_\Gamma$) as a function of the adjustable parameters.
We will look for a form similar to that used in Thm.~\ref{thm:convergence}, however there will be no error due to the observable because we are performing an eigencalculation and $\gamma_m$ will be replaced instead with an exponential decay in $\Lambda$:
\begin{equation}
    \biggl|  \epsilon_\Gamma(\Lambda,\tau) - \epsilon_\Gamma \biggr| \lesssim e^{-  \gamma_h  \tau} + e^{ - \gamma_\Lambda \Lambda}.
\end{equation}
In Fig. \ref{fig:convergence}, we focus on the momentum basis truncation radius $\Lambda$ and the interlayer tunneling truncation $\tau$.
For all twist angles and relaxation assumptions, the error decreases exponentially with both $\Lambda$ and $\tau$.
We extract the slope of this exponential convergence, $\gamma_\Lambda$ and $\gamma_h$ respectively, and study their dependence on the twist angle $\theta$.
In general, there are two ranges for the $\theta$-dependence of both $\gamma$ values: above and below the magic angle ($\theta = 1.1^\circ$).
Above the magic angle, the electronic structure is only weakly affected by the \refchange{relaxation} pattern, while at or below the magic angle the moir\'e pattern and atomic relaxations become increasingly more important to the low-energy eigenvalues as $\theta$ goes to $0$.

Starting with $\Lambda$, we see that for the relaxed system, the convergence rate is roughly constant as a function of $\theta$, in agreement with our assumption in Fig. \ref{fig:extended_bz} that only a finite energy range of the momentum basis must be included to accurately reproduce the low energy band structure.
However, the unrelaxed system converges much faster with $\Lambda$ as the twist angle decreases (e.g. larger $\gamma_\Lambda$).
This difference is caused by the fact that the interlayer tunneling function in the unrelaxed system does not change with the twist angle.
So as the twist angle becomes small, a fixed truncation radius $\Lambda$ will include monolayer Bloch states at the same energies, but the number of ``hops'' needed in momentum space to reach them grows like $\theta^{-1}$ in the unrelaxed case because of its $\theta$-independent tunneling range.

Assuming these tunnelings can be considered weak matrix perturbations, each hop between momentum basis elements reduces the effect that a higher energy state will have on a low energy eigenvalue.
Therefore, for angles where many states are included within the sampled $\Lambda$ ($\theta < 1^\circ$), we see that the unrelaxed exponential convergence $\gamma_\Lambda \propto \theta^{-1}$ (the number of hops connecting the states) while the relaxed exponential convergence $\gamma_\Lambda$ is a constant, as the relaxed tunneling range grows like $\theta^{-1}$ as well.
From a computational cost perspective, the unrelaxed system can have its $\Lambda$ decreased linearly with $\theta$.
As the magnitude of the moir\'e reciprocal lattice is also proportional to $\theta$, the matrix-size for an accurate calculation does not change with $\theta$ for the unrelaxed model.
However, for the relaxed calculation $\Lambda$ must stay a constant.
The shrinking moir\'e reciprocal lattice scale means the matrix size for an accurate calculation will grow as $\theta^{-2}$ in the relaxed model.

Moving on to $\tau$, a different $\theta$ dependence on the exponential convergence $\gamma_h$ is observed. At large angles ($\theta > 1.0^\circ$), the relaxed and unrelaxed models have identical convergence properties, since the relaxation is quite weak.
The unrelaxed $\gamma_h$ smoothly approaches a finite value as $\theta$ approaches $0^\circ$, consistent with the observation that the interlayer tunneling range is independent of $\theta$ in the unrelaxed model (Figure~\ref{fig:coupling}).
In contrast, the relaxed $\gamma_h$ goes to $0$ as $\theta$ does, showing that the tunneling range of the relaxed system scales like $\theta^{-1}$.

For extremely small twist angles, accurate calculation of relaxed TBG's electronic band structure therefore requires increasingly higher scattering frequencies in its Fourier decomposition.
This matches the reconstruction of the atomic geometry, which forms domain-walls of constant $10$ nm width \cite{relaxphysics18} that can only be described by an infinite number of Fourier components as $\theta$ goes to $0^\circ$.
As mentioned in Section~\ref{sec:intro}, this necessary inclusion of higher momentum components in the interlayer tunneling function at small angles prevents any mapping of the realistic TBG model~\cite{Carr2019exact} onto the theoretically important chiral symmetric model~\cite{Tarnopolsky2019,Khalaf2019,Wang2021}.
We note that the relaxed $\gamma_h(\theta)$ appears to reach zero at a finite value of $\theta$.
This is caused by the relatively small effects of the finite sampling of the realspace mesh of interlayer tunnelings, especially at small angles when the atomic relaxation is severe.
We found the effective intercept increased when that mesh was made larger, suggesting it is the source of the error.
Due to memory constraints of the one-shot Fourier transform we have implemented, calculations with large $\tau$ and very fine $r$-meshes were not possible.
This constraint means at large $\tau$ a residual $r$-mesh error appears in the $\tau$ convergence, affecting the estimation of the slope $\gamma_h$.

\section{Proofs}
\label{sec:proofs}

\subsection{Proof of ergodic unfolding: configuration to real and momentum to reciprocal spaces}
\label{proof:ergodicunfolding}
In this section, we prove the following theorem:
\begin{thmpf}
For $\hu \in \UnderlyingSpace^\cf_\herm$, we have
\begin{align}
\label{e:cf_rl}
&\blochmap_{\recip_b} \bigl(\Gen^\cf(\hu)\bigr) = \Gen^\rl(t_b\hu), \\
&\blochmap_{\tilde\recip_q}\bigl( \Gen^\ms(\tilde \hu)\bigr) = \Gen^\rp(t_q\tilde\hu).
\label{e:ms_rp}
\end{align}
Suppose $O^\cf \in \OpConfig^\cf$ is constructed from the set $(g, \Gen^\cf(\hu_1),\cdots \Gen^\cf(\hu_n))$, $O_b^\rl \in \OpConfig^\rl$ is constructed from the set $(g, \Gen^\rl(t_b\hu_1),\cdots \Gen^\rl(t_b\hu_n))$, $O^\ms \in \OpConfig^\ms$ is constructed from the set $(g, \Gen^\ms(\tilde \hu_1),\cdots \Gen^\ms(\tilde \hu_n))$, and $O_q^\rp \in \OpConfig^\rp$ is constructed from the set $(g, \Gen^\rp(t_q\tilde \hu_1),\cdots \Gen^\rp(t_q\tilde \hu_n))$. Then
\begin{align}
\label{e:cf_rl_operator}
& \blochmap_{\recip_b}( O^\cf) = O_b^\rl, \\
& \blochmap_{\recip_q}( O^\ms) = O_q^\rp.
\label{e:ms_rp_operator}
\end{align}
\end{thmpf}
\begin{proof}
We first set out to show 
\[
\recip_b \Gen^\cf(\hu) = \Gen^\rl(t_b\hu) \recip_b,
\]
which is sufficient to show \eqref{e:cf_rl}. Consider $\psi \in \X^\cf$. Then for $R\alpha \in \Omega_1$,
\begin{equation*}
\begin{split}
&[\recip_b\Gen^\cf(\hu)\psi]_{R\alpha} = \recip_b \begin{pmatrix} \sum_{R_1 \in \R_1} [\hu_{11}]_{R_1}(b')\psi_1(b'-R_1) + \sum_{R_2 \in \R_2}[\hu_{12}](b'-R_2)\psi_2(-b'+R_2) \\ \sum_{R_2\in\R_2} [\hu_{22}]_{R_2}(\tilde b) \psi_2(\tilde b-R_2) + \sum_{R_1' \in \R_1} [\hu_{21}](\tilde b-R_1)\psi_1(-\tilde b+R_1)\end{pmatrix}_{R\alpha} \\
&= \biggl(\sum_{R_1 \in \R_1}[\hu_{11}]_{R_1}(b+R)\psi_1(b+R-R_1) + \sum_{R_2 \in \R_2}[\hu_{12}](b+R-R_2)\psi_2(-b-R+R_2)]\biggr)_\alpha \\
&= \biggl(\sum_{R_1 \in \R_1}[\hu_{11}]_{R-R_1}(b+R)\psi_1(b+R_1) + \sum_{R_2 \in \R_2}[\hu_{12}](b+R-R_2)\psi_2(-b+R_2)]\biggr)_\alpha \\
&= \biggl(\Gen^\rp(t_b\hu)\recip_b\psi\biggr)_{R\alpha}.
\end{split}
\end{equation*}
Likewise for $R\alpha \in \Omega_2$, we have
\begin{equation*}
\begin{split}
&[\recip_b\Gen^\cf(\hu)\psi]_{R\alpha}  \\
&=\biggr(\sum_{R_2\in\R_2} [\hu_{22}]_{R_2}(-b+R) \psi_2(-b+R-R_2) + \sum_{R_1 \in \R_1} [\hu_{21}](-b+R-R_1)\psi_1(b-R+R_1)\biggr)_\alpha\\
&=\biggr(\sum_{R_2\in\R_2} [\hu_{22}]_{R-R_2}(-b+R) \psi_2(-b+R_2) + \sum_{R_1 \in \R_1} [\hu_{21}](-b+R-R_1)\psi_1(b+R_1)\biggr)_\alpha\\
&= \biggl(\Gen^\rp(t_b\hu)\recip_b\psi\biggr)_{R\alpha}.
\end{split}
\end{equation*}
This verifies \eqref{e:cf_rl}. Next we work to show
\[
\tilde\recip_q \Gen^\ms(\tilde \hu) = \Gen^\rp(t_q\tilde \hu)\tilde\recip_q.
\]
We proceed as above. We consider $\psi \in \X^\ms$, and write for $G\alpha \in \Omega_1^*$:
\begin{equation*}
\begin{split}
& \biggl(\tilde\recip_q \Gen^\ms(\tilde \hu)\psi\biggr)_{G\alpha} = \tilde \recip_q \begin{pmatrix} \sum_{G_2 \in \R_2^*}[\tilde \hu_{11}]_{G_2}(q')\psi_1(q'-G_2) + \sum_{G_1 \in \R_1^*} [\tilde \hu_{12}](q' + G_1)\psi_2(q'+G_1) \\ \sum_{G_1 \in \R_1^*} [\tilde \hu_{22}]_{G_1}(\tilde q ) \psi_2(\tilde q-G_1) + \sum_{G_2 \in \R_2^*} [\tilde \hu_{21}](\tilde q+G_2)\psi_1(\tilde q+G_2)\end{pmatrix} \\
&=\sum_{G_2 \in \R_2^*}[\tilde \hu_{11}]_{G_2}(q+G )\psi_1(q+G-G_2) + \sum_{G_1 \in \R_1^*} [\tilde \hu_{12}](q+G + G_1)\psi_2(q + G+G_1) \\
&=\sum_{G_2 \in \R_2^*}[\tilde \hu_{11}]_{G-G_2}(q+G )\psi_1(q+G_2) + \sum_{G_1 \in \R_1^*} [\tilde \hu_{12}](q+G + G_1)\psi_2(q +G_1) \\
&= \biggl(\Gen^\rp(t_q\tilde \hu)\tilde \recip_q\psi \biggr)_{G\alpha}.
\end{split}
\end{equation*}
By symmetry, the same is true for $G\alpha \in \Omega_2^*$, and \eqref{e:ms_rp} is verified. To show \eqref{e:cf_rl_operator}, consider $\psi \in \X^\cf$ and $\phi = (z-\Gen^\cf(\hu))\psi \in \X^\cf$. Then
\begin{equation*}
\recip_b \phi = (z- \Gen^\rl(t_b\hu))\recip_b \psi,
\end{equation*}
and hence
\[
(z-\Gen^\rp(t_b\hu))^{-1}\recip_b\phi = \recip_b (z-\Gen^\cf(\hu))^{-1}\phi.
\]
In other words,
\[
\blochmap_{\recip_b}\bigl((z-\Gen^\cf(\hu))^{-1}\bigr) = (z-\Gen^\rl(t_b\hu))^{-1}.
\]
Now we have 
\[
O^\cf = \int_\contourprod g(z) \prod_{j=1}^n(z_j - \Gen^\cf(\hu))^{-1}dz.
\]
By the resolvent relation above we have
\[
\blochmap_{\recip_b}(\prod_{j=1}^n(z_j - \Gen^\cf(\hu))^{-1}) = \prod_{j=1}^n(z_j - \Gen^\rl(t_b\hu))^{-1}.
\]
By continuity in $z$, we obtain \eqref{e:cf_rl_operator}.
The same argument yields \eqref{e:ms_rp_operator}.

\end{proof}

\subsection{Proof of Bloch unitary mapping: real space to momentum and reciprocal  to configuration spaces}
\label{proof:blochtransform}

\begin{thmpf}
For $\hu \in \UnderlyingSpace^\cf_\herm$, we have
\begin{align}
\label{e:rl_ms}
& \blochmap_{\G} \bigl( \Gen^\rl(\hu) \bigr) = \Gen^\ms(\tilde \hu), \\
& \blochmap_{\tilde\G} \bigl( \Gen^\rp(\tilde \hu)\bigr) = \Gen^\cf(\hu).
\label{e:rp_cf}
\end{align}
Suppose $O^\cf \in \OpConfig^\cf$ is constructed from the set $(g, \Gen^\cf(\hu_1),\cdots \Gen^\cf(\hu_n))$, $O^\rl \in \OpConfig^\rl$ is constructed from the set $(g, \Gen^\rl(\hu_1),\cdots \Gen^\rl(\hu_n))$, $O^\ms \in \OpConfig^\ms$ is constructed from the set $(g, \Gen^\ms(\tilde \hu_1),\cdots \Gen^\ms(\tilde \hu_n))$, and $O^\rp \in \OpConfig^\rp$ is constructed from the set $(g, \Gen^\rp(\tilde \hu_1),\cdots \Gen^\rp(\tilde \hu_n))$. Then
\begin{align}\
\label{e:rl_ms_operator}
& \blochmap_{\G}( O^\rl)= O^\ms \\
& \blochmap_{\tilde \G}( O^\rp) = O^\cf.
\label{e:rp_cf_operator}
\end{align}
\end{thmpf}
\begin{proof}
As in the previous proof, we focus on proving the \dm{isomorphic} relation
\[
\G \Gen^\rl(\hu) = \Gen^\ms(\tilde \hu)\G.
\]
Consider $\psi \in \X^\rl$. For $R \in \R_j$, we denote $\psi_R$ as the vector of size $|\A_j|$ of orbitals corresponding to site $R$. We will make use of the Poisson summation formula:
\[
\sum_{R \in \R_2}e^{i\xi \cdot R} = |\Gamma_2^*| \sum_{G \in \R_2^*}\delta(\xi - G).
\]
Using the definitions provided and the Poisson summation formula, we calculate
\begin{equation*}
\begin{split}
& [\G \Gen^\rl(\hu)\psi]_1(q) = |\Gamma_1^*|^{-1/2}\sum_{R_1 \in \R_1}e^{-iq\cdot R_1} \biggl(\sum_{R'_1 \in \R_1}[\hu_{11}]_{R_1-R_1'}(R_1)\psi_{R_1'} + \sum_{R_2'\in\R_2^*} [\hu_{12}](R_1-R_2') \psi_{R_2'} \biggr)\\
& = |\Gamma_1^*|^{-1/2}\sum_{R_1 \in \R_1}e^{-iq\cdot R_1} \biggl(\sum_{R'_1 \in \R_1,G_2 \in \R_2^*}[\hu_{11}]_{R_1-R_1',G_2}e^{iG_2\cdot R_1}\psi_{R_1'} + \sum_{R_2'\in\R_2} \int_\xi [\hat\hu_{12}](\xi)e^{i\xi\cdot(R_1-R_2')}d\xi \psi_{R_2'} \biggr) \\
&=  \sum_{R_1 \in \R_1,G_2 \in \R_2^*}[\hu_{11}]_{R_1G_2}e^{i(G_2-q)\cdot R_1}\G_1\psi_1(q-G_2) + \sum_{R_1 \in \R_1} \frac{|\Gamma_2^*|^{1/2}}{|\Gamma_1^*|^{1/2}}\int_\xi [\hat\hu_{12}](\xi)e^{i(\xi-q) \cdot R_1} \G_2\psi_2(\xi) d\xi \\
&=\sum_{G_2 \in \R_2^*}[\tilde\hu_{11}]_{G_2}(q)\G_1\psi_1(q+G_2) + c_1^*c_2^*\sum_{G_1 \in \G_1^*} \int_\xi [\hat\hu_{12}](q+G_1) \G_2\psi_2(q+G_1) \\
&= [\Gen^\ms(\tilde \hu)\G\psi]_1.
\end{split}
\end{equation*}
The same argument holds for the second component, which verifies \eqref{e:rl_ms}. Next we consider
\[
\tilde \G \Gen^\rp(\tilde \hu) = \Gen^\cf(\hu)\tilde \G.
\]
Then we have for $\psi \in \X^\rp$
\begin{equation*}
\begin{split}
& [\tilde \G \Gen^\rp(\tilde \hu)\psi]_1(b) =  \sum_{G_2 \in \R_2^*} \frac{e^{i b\cdot G_2}}{|\Gamma_2|^{1/2}} \biggl(\sum_{G_2' \in \R_2^*} [\tilde \hu_{11}]_{G_2-G_2'}(G)\psi_{G_2'} + \sum_{G_1 \in \R_1^*}[\tilde \hu_{12}](G_2+G_1)\psi_{G_1}\biggr) \\
& =  \sum_{G_2 \in \R_2^*} \frac{e^{i b\cdot G_2}}{|\Gamma_2|^{1/2}} \biggl(\sum_{R_1\in\R_1,G_2' \in \R_2^*} [\tilde \hu_{11}]_{G_2-G_2',R_1}e^{iR_1\cdot G_2}\psi_{G_2'} + \sum_{G_1 \in \R_1^*}\frac{c_1^*c_2^*}{(2\pi)^2}\int_{\mathbb{R}^2}[ \hu_{12}](x)e^{-ix\cdot(G_2+G_1)}\psi_{G_1}dx\biggr) \\
&=   \biggl(\sum_{R_1\in\R_1,G_2' \in \R_2^*} [ \hu_{11}]_{R_1,G_2'}e^{iG_2'\cdot b}\tilde \G\psi_1(b+R_1) + \frac{1}{|\Gamma_2^*|}\sum_{G_2 \in \R_2^*}\int_{\mathbb{R}^2}[ \hu_{12}](x)e^{i(b-x)\cdot G_2}\tilde \G\psi_2(-x)dx\biggr) \\
& \biggl(\sum_{R_1\in\R_1} h_{R_1}(b)\tilde \G\psi_1(b+R_1) + \sum_{R_2 \in \R_2}[ \hu_{12}](b-R_2)\tilde \G\psi_2(-b+R_2)\biggr) \\
&= [\Gen^\cf(\hu)\tilde \G \psi]_1.
\end{split}
\end{equation*}
The same argument holds for the second component by symmetry, and thus we have \eqref{e:rp_cf}. As in the previous theorem, \eqref{e:rl_ms_operator} and \eqref{e:rp_cf_operator} follow by the transformation of the resolvent, and then continuity of the resolvent with respect to $z$. 

\end{proof}

\subsection{Proof of configuration and momentum space operator representations}
\label{proof:representation}

\begin{thmpf}
Suppose $O^\cf \in \OpConfig^\cf$ is constructed from the set $(g, \Gen^\cf(\hu_1),\cdots \Gen^\cf(\hu_n))$, $O_b^\rl \in \OpConfig^\rl$ is constructed from the set $(g, \Gen^\rl(t_b\hu_1),\cdots \Gen^\rl(t_b\hu_n))$, $O^\ms \in \OpConfig^\ms$ is constructed from the set $$(g, \Gen^\ms(\tilde \hu_1),\cdots \Gen^\ms(\tilde \hu_n)),$$ and $O_q^\rp \in \OpConfig^\rp$ is constructed from the set $(g, \Gen^\rp(t_q\tilde \hu_1),\cdots \Gen^\rp(t_q\tilde \hu_n))$. Then
\begin{align}
\label{e:cf_integral}
& \Tr \; O^\cf = \nu \biggl(\sum_{\alpha \in \A_1} \int_{\Gamma_2} [O_b^\rl]_{0\alpha,0\alpha}db + \sum_{\alpha \in \A_2} \int_{\Gamma_1} [O_b^\rl]_{0\alpha,0\alpha}db\biggr), \\
& \Tr \; O^\ms = \nu^* \biggl( \sum_{\alpha \in \A_1} \int_{\Gamma_1^*} [O_q^\rp]_{0\alpha,0\alpha}dq + \int_{\Gamma_2^*} [O_q^\rp]_{0\alpha,0\alpha}dq\biggr).
\label{e:ms_integral_pf}
\end{align}
\end{thmpf}
\begin{proof}
{By definition of the trace, 
\begin{equation*}
\Tr \; O^\cf = \nu \biggl( \int_{\Gamma_2}\tr \; [\hu_{11}]_0(b)db + \int_{\Gamma_1}\tr\; [\hu_{22}]_0(b)db\biggr).
\end{equation*}}
Then \eqref{e:cf_integral} is verified by observing the relation {
\[
[\hu_{jj}(b)]_0 = \biggr[(\recip_b O^\cf \recip_b^*)_{0\alpha,0\alpha'}\biggr]_{\alpha,\alpha' \in \A_j}.
\]}
The proof for \eqref{e:ms_integral_pf} is the same.

\end{proof}

\subsection{Proof of equivalence of observables in all space}
\label{proof:observables}

\begin{thmpf}
For $O^\name \in \OpConfig^\name$ with hopping functions $\hu \in \UnderlyingSpace^\cf_\herm$ for real and configuration spaces, and $\tilde \hu$ for reciprocal and momentum spaces, we have
\begin{equation}
\TrLim \; O^\rp = \Tr \; O^\ms = \TrLim \; O^\rl = \Tr \; O^\cf.
\end{equation}
\end{thmpf}
\begin{proof}
We first verify $\TrLim \; O^\rl = \Tr \; O^\cf$, which comes from observing the thermodynamic limit trace of real space is an ergodic sampling of configuration space. Let
\[
L_\alpha(b) = [O_b^\rl]_{0\alpha,0\alpha}.
\]
Then $[O^\rl]_{R\alpha,R\alpha} = L_\alpha(R)$. And we obtain
\begin{equation*}
\begin{split}
\TrLim \; O^\rl &= \lim_{r \rightarrow \infty} \frac{1}{\# \Omega_r}\sum_{R\alpha \in \Omega_r} L_\alpha(R) \\
&= \lim_{r \rightarrow \infty}\biggl( \frac{\#\Omega_r \cap \Omega_1}{\# \Omega_r} \sum_{R\alpha \in \Omega_r \cap \Omega_1} L_\alpha(R) + \frac{\#\Omega_r \cap \Omega_2}{\# \Omega_r} \sum_{R\alpha \in \Omega_r \cap \Omega_2} L_\alpha(R)\biggr) \\
& = \nu \biggl( \sum_{\alpha \in \A_1}\int_{\Gamma_2} L_\alpha(b)db + \sum_{\alpha \in \A_2} \int_{\Gamma_1} L_\alpha(b)db \biggr),
\end{split}
\end{equation*}
where the last line follows from the ergodic theorem, as in Theorem 2.1 of \cite{massatt2017}. The equivalence of observables in real and configuration space is concluded by using Theorem \ref{thm:representation}. The proof of 
\[
\Tr \; O^\ms = \TrLim\; O^\rp
\]
is the same. We then focus \refchange{on} proving the last needed equality,
\[
\Tr \; O^\ms = \TrLim \; O^\rl.
\]
We let $e_{R\alpha}$ be the standard basis vector in $\X^\rl$ and $e_\alpha$ the standard basis vector in $\mathbb{C}^{\mathcal{A}_j}$ for $\alpha \in \A_j$. We denote $\tilde \recip_j$ acting on $\X_j^\ms$ as 
\[
\tilde \recip_j \psi = \{ \psi(G)\}_{G \in \R_i^*}
\]
for $i \neq j$. Observe that if $\phi, \psi \in \X_j^\ms$, then
\[
\dm{|\Gamma_j^*|^{-1/2}}\langle \phi,  \psi \rangle = \TrLim \; (\tilde\recip_j \phi)^* (\tilde\recip_j\psi).
\]
where $\TrLim$ here is understood as only being computed over $\Omega_j^*$, and $\langle \cdot, \cdot \rangle$ is the standard $L^2$ inner product over $\X_j^\ms$.

For $R\alpha \in \Omega_1$ (without loss of generality), we note
\begin{equation*}
\begin{split}
[ O^\rl]_{R\alpha,R\alpha} &= (\G e_{R\alpha})^* \G O^\rl \G^* (\G e_{R\alpha}) \\
&= |\Gamma_1^*|^{-1} \int_{\Gamma_1^*} e^{iR\cdot q} e_\alpha^*O^\ms e_\alpha e^{-i R\cdot q}dq \\
&= \TrLim \;  (\tilde\recip_1(e_\alpha e^{iR\cdot q}))^*  (\tilde \recip_1 O^\ms e_\alpha e^{-i R\cdot q}) \\
&= \TrLim \;  (\tilde\recip_1(e_\alpha e^{iR\cdot q}))^*  ( O^\rp \tilde \recip_1 e_\alpha e^{-i R\cdot q}) \\
&= \lim_{r\rightarrow \infty}\frac{1}{\# \{ G \in \R_2^* : |G| < r\}}\sum_{G,G' \in \R_2^*, |G| < r} e^{i R\cdot(G-G')} O^\rp_{G\alpha,G'\alpha} \\
&= \sum_{G \in \R_2^*} e^{-i R\cdot G} \mint_{\Gamma_1^*}[O_q^\rp]_{0\alpha,G\alpha}dq.
\end{split}
\end{equation*}
The last equality follows from
\[
[O_{G}^\rp]_{0\alpha,(G'-G)\alpha} = [O^\rp]_{G\alpha,G'\alpha}
\]
and the ergodic theorem in Theorem 2.1 \cite{massatt2017}.
To understand $\TrLim \; O^\rl$ then, we observe 
\begin{equation*}
\lim_{r\rightarrow \infty} \frac{1}{\# \{ R \in \R_1: |R| < r\} } \sum_{R \in \R_1: \; |R| < r} e^{-i R \cdot G}  = \delta_{G0} .
\end{equation*}
We then obtain
\begin{equation*}
\begin{split}
\TrLim \; O^\rl &= \lim_{r\rightarrow \infty} \frac{1}{\# \Omega_r}\sum_{j=1}^2 \sum_{G\alpha \in \Omega_j^*}   \sum_{R \in \R_j :\; |R| < r} e^{-i R \cdot G} \mint_{\Gamma_j^*}[O^\rp_q]_{0\alpha,G\alpha}dq\\
&= \nu^*\sum_{j=1}^2\sum_{\alpha \in \A_j}  \int_{\Gamma_j^*} [O_q^\rp]_{0\alpha,0\alpha}dq \\
&= \Tr \; O^\ms.
\end{split}
\end{equation*}
This concludes the proof.

\end{proof}

\subsection{Proof of incommensurate Brillouin Zone representation}
\label{proof:prop}
\begin{proppf}
If $G_j = (-1)^{j}2\pi A_j^{-T} n$ for $n \in \mathbb{Z}^2$ and
\begin{equation}
\label{e:prop_eq}
\delta q = \Theta_{21}n = 2\pi (A_2^{-T} - A_1^{-T})n,
\end{equation}
then \dm{
\[
[H_r(q)]_{(G_j+G_j')\alpha,(G_\ell + \tilde G_\ell)\alpha'} = [H_r(q+\delta q)]_{G_j'\alpha,\tilde G_\ell\alpha'}, \hspace{2cm}G_j' \in \R_j^*, \; \tilde G_\ell \in \R_\ell^*
\]}
as long as $q, q+\delta q \in {\Gamma_j^*(\energy+B_\eta)+B_r}$ for $j = 1,2$.
\end{proppf}
\begin{proof}
For intralayer coupling of sheet 1, starting from the left-hand side of \eqref{e:prop_eq} we obtain
\begin{equation*}
\begin{split}
[H_r(q)]_{(G_2 + G_2')\alpha,(G_2 + \tilde G_2)\alpha'} &= [\tilde \hu^{(\tau)}_{11}]_{G_2'-\tilde G_2,\alpha\alpha'}(q + G_2 + G_2') \\
&= [\tilde \hu^{(\tau)}_{11}]_{G_2'-\tilde G_2,\alpha\alpha'}(q + \delta q + G_2') \\
&= [H_r(q+\delta q)]_{G_2'\alpha, \tilde G_2\alpha'}.
\end{split}
\end{equation*}
The case of intralayer sheet 2 is identical.
We next consider interlayer coupling from sheet 2 to sheet 1:
\begin{equation*}
\begin{split}
[H_r(q)]_{ (G_2 + G_2')\alpha,(G_1 + \tilde G_1)\alpha'} &= [\tilde \hu_{12}^{(\tau)}]_{\alpha\alpha'}(q+G_2 + G_2' + G_1 + \tilde G_1) \\
&=[\tilde \hu_{12}^{(\tau)}]_{\alpha\alpha'}(q+\delta q + G_2'  + \tilde G_1) \\
&= [H_r(q+\delta q)]_{G_2'\alpha,\tilde G_1\alpha'}.
\end{split}
\end{equation*}

\end{proof}

\subsection{Proof of numerical convergence rate}
\label{proof:convergence}

\begin{thmpf}
Consider incommensurate bilayer system as described above with long moir\'{e} length scale, \refchange{i.e. using Assumption \ref{assump:small}}. Consider $E \in \energy$, and $\varepsilon \ll 1$. Let $\tau > 0$ be a hopping truncation. Then there are constants $\gamma_h$, $\gamma_m$, and $\gamma_g$ corresponding to hopping truncation error, momenta truncation error, and Gaussian decay rates respectively such that
\begin{equation}
\label{e:error_bounds}
    \biggl| D_\varepsilon(E) - D_{\varepsilon,r}(E) \bigg| \lesssim \varepsilon^{-3/2}(e^{-  \gamma_h  \tau} + \varepsilon^{-2}e^{ - \gamma_m r}+e^{- \gamma_g \varepsilon^{-2}} ) 
\end{equation}
where 
\[
D_{\varepsilon,r}(E) = \nu^*\int_{\Gamma_{21}^*} \Tr \; \phi_\varepsilon(E-H_r(q))dq.
\]
When mechanical relaxation effects are not included, i.e.
\begin{equation}
\Gen^\ms(\tilde \hu) \in \OpConfig^\ms(\tilde \gamma,  \gamma),
\end{equation}
then we have $\gamma_h$ is independent of $\theta$, but $ \gamma_m = O( \theta^{-1})$. Meanwhile if mechanical relaxation effects are included, i.e. 
\begin{equation}
\Gen^\ms(\tilde \hu) \in \OpConfig^\ms(\tilde \gamma \theta,  \gamma),
\end{equation}
then we have $\gamma_h = O(\theta)$ and $\gamma_m = O(1)$.
\end{thmpf}
\begin{proof}
We begin by quantifying truncation error. We observe that
\[
\| \Gen^\rp(\tilde \hu) - \Gen^\rp(\tilde \hu^{(\tau)})\|_\op \lesssim e^{-\gamma_h \tau}.
\]
We define
\[
D_\varepsilon^{(\tau)}(E) = \TrLim\; \phi_\varepsilon(E-\Gen^\rl(\tilde \hu^{(\tau)})).
\]
We note as long as $\tau$ is sufficiently large, we have
\[
|D_\varepsilon(E) - D_\varepsilon^{(\tau)}(E)| \lesssim \max_{E'}|\phi_\varepsilon'(E')|e^{-\gamma_h \tau} \lesssim \varepsilon^{-3/2}e^{-\gamma_h\tau}.
\]
By Proposition \ref{prop:inc_BZ}, we have
\[
D_{\varepsilon,r}(E) = \dm{\sum_{k=1}^n\nu^*\int_{\Gamma_{21^*}+q_k}} \Tr \; \phi_\varepsilon(E-H_r(q))dq = \nu^* \sum_{j=1}^2 \sum_{\alpha \in \A_j} \int_{\Gamma_j^*} [\phi_\varepsilon(E-H_r(q))]_{0\alpha,0\alpha}dq.
\]
{We note that we can expand the integral above to all momenta in $\Gamma_j^*$ as $H_r(q)$ is the empty matrix for $q \not\in {\Gamma_j^*(\energy+B_\eta)+B_r}$, in which case we consider $[\phi_\varepsilon(E-H_r(q)]_{0\alpha,0\alpha} =0$.}
We thus proceed with the right-hand side. 
For simplicity of notation, we denote for matrices $B(q)$
\[
T(B) = \nu^* \sum_{j=1}^2 \sum_{\alpha \in \A_j} \int_{\Gamma_j^*} [B(q)]_{0\alpha,0\alpha}dq.
\]
So in particular,
\[
D_{\varepsilon,r}(E) = T(\phi_\varepsilon(E-H_r)).
\]
Since we expect different energies to contribute differently to error, we construct a contour $C$ around the spectrum of the Hamiltonian such that $d(C, \Gen^\rp(\tilde \hu)) \in (\varepsilon,2\varepsilon)$. If the spectrum has gaps, then $C$ would not be a simple curve in the complex \refchange{plane} but a union of one per ungapped interval of spectrum. We next divide $C$ into two regions, 
\begin{align*}
& C_+ = \{ z \in C : \Real(z) \in \energy + {B_{\eta'}} \}, & C_- \in C \setminus C_+.
\end{align*}
{Here $\eta' = \frac{\alpha \eta}{2(2+\alpha)}$}
We observe for $z \in C_-$, there is a $\gamma_g > 0$ such that
\[
|\phi_\varepsilon(E - z)|  \lesssim \varepsilon^{-1/2} e^{-\gamma_g \varepsilon^{-2}}.
\]
We observe
\begin{equation*}
\begin{split}
D_{\varepsilon,r}(E) &= \frac{1}{2\pi i} \oint_Cg(z) T( (z-H_r)^{-1})dz \\
&= \frac{1}{2\pi i} \biggl(\oint_{C_+}g(z) T( (z-H_r)^{-1})dz + \oint_{C_-}g(z) T( (z-H_r)^{-1})dz\biggr).
\end{split}
\end{equation*}
Likewise
\begin{equation*}
\begin{split}
D_{\varepsilon}^{(\tau)}(E) &= \frac{1}{2\pi i} \oint_Cg(z) T( (z-\Gen^\rp(t_{(\cdot)}\tilde \hu^{(\tau)}))^{-1})dz \\
&= \frac{1}{2\pi i} \biggl(\oint_{C_+}g(z) T( (z-\Gen^\rp(t_{(\cdot)}\tilde \hu^{(\tau)}))^{-1})dz + \oint_{C_-}g(z) T( (z-\Gen^\rp(t_{(\cdot)}\tilde \hu^{(\tau)}))^{-1})dz\biggr).
\end{split}
\end{equation*}
The second terms in both equations are bounded up to a constant by $\varepsilon^{-3/2}e^{-\gamma_g \varepsilon^{-2}}$, completing the Gaussian tail error term in \eqref{e:error_bounds}. To complete the error bound, it suffices to show
\[
\biggl\| \oint_{C_+} g(z)\bigl( T( (z-\Gen^\rp(t_{(\cdot)}\tilde \hu^{(\tau)}))^{-1}) - T( {(z-H_r)^{-1}}\bigr)dz\biggr\|_\op \lesssim \varepsilon^{-3/2}e^{-\gamma_m r}.
\]
Observation of the operators shows it is sufficient to prove for arbitrary $z \in C_+$ and momenta $q \in \Gamma_1^*$ (without loss of generality) that
\begin{equation}
\label{e:resolvent_error}
\biggl\| [(z-\Gen^\rp(t_q\tilde \hu^{(\tau)})^{-1}]_{0\alpha,0\alpha} - [(z - H_r(q))^{-1}]_{0\alpha,0\alpha} \biggr\|_\op \lesssim \varepsilon^{-1}e^{-\gamma_m r}.
\end{equation}
Here $\alpha \in \A_1$. The principle technique here is a ring decomposition. We will define an increasing collection of radii $r_0,\cdots, r_n$ such that $r_0 = 0$ and $r_n = r$. We write $H = \Gen^\rp(\tilde \hu^{(\tau)})$. We have the following decomposition:
\begin{align*}
& U_0 = \Omega_{r_0}^*(q), \\
& U_j = \Omega_{r_j}^*(q) \setminus \Omega_{r_{j-1}}^*(q), \; j > 0, \\
& J_{j} = J_{\Omega^* \leftarrow U_j},\\
& H_{ij} = J_i^* H J_j.
\end{align*}
We assume one final ring denoted `$\infty$' that corresponds to remaining degrees of freedom, i.e.
\[U_\infty = \Omega^* \setminus \Omega_r(q).\] 
We choose $n$ and $r_j = j/n$ in such a fashion that $H_{ij} = 0$ if $|i-j| > 1$ so that the rings form a ``nearest neighbor" type decomposition (see Figure \ref{fig:rings}), which can be achieved as sites couple in a distance $\tau$. This can be achieved for $n$ proportional to $\theta \tau$ with correctly chosen proportionality constant, since distance in momenta is on the inverse moir\'{e} scale while $\tau$ is on the lattice scale. {For simplicity of notation, we assume $n$ is divisible by $4$.} We observe
\begin{equation*}
H = \begin{pmatrix} H_{00} & H_{01} & 0 & 0 & \cdots  \\ H_{10} & H_{11} & H_{12}  & 0 &\cdots \\
0 & H_{21} & H_{22} & H_{23} & \cdots  \\\vdots & \ddots & \ddots & \ddots & \ddots  \\
0&  0 & \cdots & H_{\infty,n} & H_{\infty,\infty}\end{pmatrix}
\end{equation*}
\begin{figure}[ht]
\centering
\includegraphics[width=.45\textwidth]{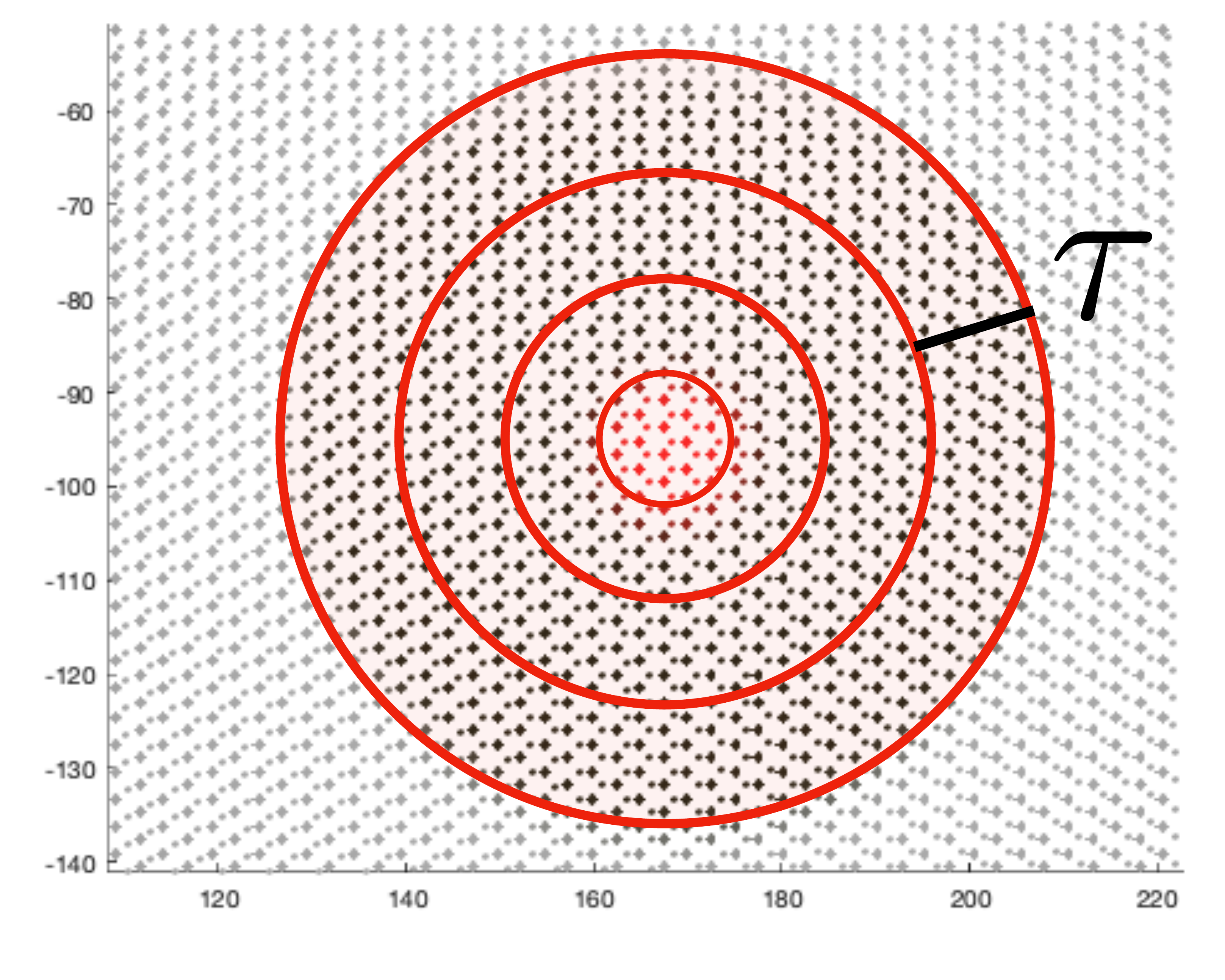}
\caption{The rings correspond to the sets $U_j$. $\R_1^* \cup \R_2^*$ are displayed.}
\label{fig:rings}
\end{figure}
We let $H_{i\leftrightarrow j}$ correspond to the matrix restricted to the rings $i$ through $j$ for $i < j$. As a slight abuse of notation we also use 
\[J_i = J_{\cup_{j=\ell}^m U_j \leftarrow U_i}\]
where the choice of definition of $J_i$ will be clear from the context. 
With our ring decomposition fully constructed, we now employ Schur complement techniques to prove \eqref{e:resolvent_error}. We use the following resolvent notation:
\begin{align*}
& R_{i \leftrightarrow j} = (z - H_{i\leftrightarrow j})^{-1}, & R_j = (z-H_{jj})^{-1}, \hspace{1.5cm} & R = (z-H)^{-1}.
\end{align*}
We also denote the natural injection for the full approximation
\[
J = J_{\Omega^* \leftarrow \Omega^*_r(q) }.
\]
We recall the general Schur complement formula  for matrices $A, B, C, D$, $M = \begin{pmatrix} A & B \\ C & D\end{pmatrix}$, and $M / A := D - CA^{-1}B$ is
\begin{equation}
\label{e:schur}
M^{-1} = \begin{pmatrix} A^{-1} + A^{-1}B(M/A)^{-1}CA^{-1} & -A^{-1}B(M/A)^{-1} \\ -(M/A)^{-1} CA^{-1} & (M/A)^{-1}\end{pmatrix}.
\end{equation}
{We denote the ring that $q$ lives on as $k \in \{0,\cdots n,\infty\}$. First we consider $k \leq n /2$.}
In the newly constructed notation, we observe, using an application of Schur complement for the ring decomposition via the formula for $M_{11}^{-1}$ above,
{\begin{equation*}
\begin{split}
\biggl\| J_k^* &(z-H)^{-1} J_k - J_k^*(z - H_{0 \leftrightarrow n})^{-1}J_k \biggr\|_\op = \biggl\| J_k^* R J_k - J_k^*R_{0 \leftrightarrow n}J_k \biggr\|_\op \\
&= \biggl\| J_k^*\bigl(R_{0\leftrightarrow n}+R_{0\leftrightarrow n}J^*H J_\infty J_\infty^* R J_\infty J_\infty^*HJR_{0\leftrightarrow n}\bigr)J_k - J_k R_{0\leftrightarrow n}J_k \biggr\|_\op \\
&= \biggl\| J_k^*R_{0\leftrightarrow n}J^*H J_\infty J_\infty^* R J_\infty J_\infty^*HJR_{0\leftrightarrow n}\bigr)J_k  \biggr\|_\op \\
&\lesssim \varepsilon^{-2}\| J_k^* R_{0\leftrightarrow n}J_n \|_\op.
\end{split}
\end{equation*}}
The last inequality is found by noting $J^*HJ_\infty J_\infty^*$ only can couple ring $n$ on the left to ring $\infty$ on the right, as $J_\infty J_\infty^*$ is the projection onto the `$\infty$' ring over $\X^\rp$ and $H$ is nearest neighbor in ring coupling. Matching the Schur complement expression \eqref{e:schur} to the second line, we have
\begin{align*}
&A = H_{0\leftrightarrow n}, & B = J^*HJ_\infty \\
& C =J^*HJ_\infty, & D = J_\infty^*RJ_\infty.
\end{align*}
We rewrite ${J_k^*R_{0\leftrightarrow n}J_n}$ using the $M_{12}^{-1}$ entry in \eqref{e:schur} in an interative fashion as follows:{
\begin{equation}
\label{e:exp_bound}
    \begin{split}
    J_k^*R_{0\leftrightarrow n}J_n &= -J_k R_{0\leftrightarrow n-1} H_{n-1,n} R_n \\
    &= J_k^* R_{0 \leftrightarrow n-2}H_{n-2,n-1}R_{n-1} H_{n-1,n}R_n \\
    & \hspace{2mm}\vdots \\
    &= (-1)^{n-k}J_k^* R_{0 \leftrightarrow k} J_k\prod_{j=k+1}^n H_{j-1,j}R_j
    \end{split}
\end{equation}}
We now observe, recalling the definition of $\eta$ in $\eqref{e:eta}$:{
\begin{equation*}
\begin{split}
\| H_{j-1,j}R_j \|_\op &\lesssim \|\Gen^\rp(\tilde \hu) - \Gen^\rp(\mon) \|_\op \| (z - H_{jj})^{-1}\|_\op \\
&\leq \|\Gen^\rp(\tilde \hu) - \Gen^\rp(\mon) \|_\op \cdot \biggl( \| z - J_j^*\Gen^\rp(t_q \mon)J_j\|_\op - \\
&\hspace{1cm}\| H_{jj} - J_j^*\Gen^\rp(t_q\mon)J_j\|_\op\biggr)^{-1} \\
& \leq \frac{\beta}{1+\alpha/2}.
\end{split}
\end{equation*}
Here
\[
\beta = \frac{\|\Gen^\rp(\tilde \hu) - \Gen^\rp(\mon) \|_\op}{\| \| H_{jj} - J_j^*\Gen^\rp(t_q\mon)J_j\|_\op}.
\]
For $\alpha$ sufficiently large relative to $\beta$, we have $\beta(1+\alpha/2)^{-1} < 1$. Note as $\tau \rightarrow \infty$, $\beta \rightarrow 1$, so in practice we don't need $\alpha$ large.
}Taking an operator bound in \eqref{e:exp_bound}, we obtain
\begin{equation}
    \label{e:exp_bound_rings}
    \| J_0 R_{0\leftrightarrow n}J_n^*\|_\op \lesssim \varepsilon^{-1}\biggl({\frac{\beta}{1+\alpha/2}\biggr)^{n/2}} = \varepsilon^{-1}e^{-\lambda r}
\end{equation}
for some $\lambda >0$. Here we used $r \sim n$. The momenta cut-off term in the error is now justified {for $k \leq n/2$. Next we consider $k > n/2$.}
{
\begin{equation*}
\begin{split}
\biggl\|& J_k^*  (z-H)^{-1} J_k - J_k^*(z - H_{0 \leftrightarrow n})^{-1}J_k \biggr\|_\op \\
&= \biggl\| J_k^*R_{0\leftrightarrow n/4}J_k +  J_k^* R_{0\leftrightarrow n/4}J_{n/4}H_{n/4,n/4+1}J_{n/4+1}^*(z-H)^{-1}J_{n/4+1}H_{n/4+1,n/4}J_{n/4}^*R_{0\leftrightarrow n/4} J_k \\
&- J_k^*R_{0\leftrightarrow n/4}J_k-J_k^* R_{0\leftrightarrow n/4}J_{n/4}H_{n/4,n/4+1}J_{n/4+1}^*(z-H_{0\leftrightarrow n})^{-1}J_{n/4+1}H_{n/4+1,n/4}J_{n/4}^*R_{0\leftrightarrow n/4} J_k \biggr\|_\op\\
& \leq \eps^{-2} \| J_k^*R_{0\leftrightarrow n/4}J_{n/4}\|_\op.
\end{split}
\end{equation*}
By the same argument above, this has the same bound up to choice of $\gamma_m$ as in \eqref{e:exp_bound_rings}.}

We observe $\lambda$ informs the value of $\gamma_m$. The exact relation is not important as we know they are proportional with the constant of proportionality independent of $\theta$. The dependence of the convergence rates $\gamma_h$ on $\theta$ follow from the form of the hopping functions when mechanical relaxation is included or not, and the $\gamma_m$ dependence follows from the choice of nearest neighbor rings, so includes a $\tau$ dependence.

\end{proof}

\appendix
\section{Mechanical Relaxation Model}
\label{app:relax}

 We next define the moir\'e superlattice \cite{cazeaux2018energy} with its unit cell:
 \begin{align*}
 &\R_\M := (A_2^{-1} - A_1^{-1})^{-1}\mathbb{Z}^2, \\
 & \Gamma_\M := \{ (A_2^{-1} - A_1^{-1})^{-1}\beta \text{ : } \beta \in [0,1)^2\}.
 \end{align*}
 We can map the moir\'e supercell to configuration space by the mappings $\gamma_j : \Gamma_\M \rightarrow \Gamma_j$,
 \begin{equation*}
 \gamma_j : x \mapsto (I-A_jA_{P_j}^{-1})x.
 \end{equation*}
 \dm{Here $P_1 = 2$ and $P_2 = 1$.}
 We then define
 \begin{equation*}
 u_{\M,j}(x) := u_{j}(\gamma_{P_j}(x)).
 \end{equation*}
 Since $A_1 \approx A_2$, locally the lattice configuration looks periodic. As a consequence, the interlayer coupling energy can be approximated using a Generalized Stacking Fault Energy (GSFE) functional, $\Phi : \Gamma_j \rightarrow \mathbb{R}$.
  In particular, the interlayer energy can be shown to be well modeled by \cite{relaxphysics18,cazeaux2018energy}:
\begin{equation*}
\sum_{j=1}^2\frac{1}{2}\mint_{\Gamma_M} \Phi(\gamma_{P_j}(x) + u_{\M,P_j}(x) - u_{\M,j}(x))dx.
\end{equation*}
This is effective because the interlayer coupling energy is assumed to be perturbative, i.e., $\|\Phi\|_\infty \ll 1$. The intralayer energy can be modeled via elasticity tensors, $\varepsilon_j$, $j \in \{1,2\}$. 
The intralayer energy is then given by
\begin{equation*}
\sum_{j=1}^2\int_{\Gamma_\M} \frac{\nabla u_{\M,j}+\nabla u_{\M,j}^T}{2} \cdot \varepsilon_{j} \frac{\nabla u_{\M,j} + \nabla u_{\M,j}^T}{2}dx,
\end{equation*}
and the total elastic energy functional to be minimized is then
\begin{equation*}
\begin{split}
E(u) = \sum_{j=1}^2\mint_{\Gamma_\M}\biggl(\Phi(\gamma_{P_j}(x)& + u_{\M,j}(x)-u_{\M,P_j}(x))\\
&+ \frac{1}{2} \frac{\nabla u_{\M,j}+\nabla u_{\M,j}^T}{2} \cdot \varepsilon_{j} \frac{\nabla u_{\M,j} + \nabla u_{\M,j}^T}{2}\biggr)dx.
\end{split}
\end{equation*}
If the two materials are identical, $\varepsilon_1 = \varepsilon_2$ and $\det(A_1) = \det(A_2)$, so by symmetry
\begin{equation*}
u_{\M,1} = \frac{1}{2}u_\M, \qquad
u_{\M,2} = -\frac{1}{2}u_\M,
\end{equation*}
where $u_\M = u_{\M,1} -u_{\M,2}$.

Let $h_{ij}$ be the coupling tight-binding functionals defined via the distance between lattice sites.
We shall focus on twisted bilayer graphene as a case study, so we will use this symmetry in the numerics. 

\bibliographystyle{abbrv}
\bibliography{relax}

\end{document}